\documentclass[a4paper, 11pt, twoside, notitlepage,reqno]{amsart}
\usepackage{ragged2e}
\usepackage[colorlinks=true,urlcolor=magenta,linkcolor=red,citecolor=blue]{hyperref}
\usepackage[usenames]{color}
\usepackage{epsfig, graphics, graphicx}
\usepackage{cleveref}
\usepackage{upref}
\usepackage[T1]{fontenc}    
\usepackage[utf8]{inputenc} 
\usepackage{lmodern} 
\usepackage{subcaption, comment, bm}
\usepackage[percent]{overpic}
\usepackage{float}
\usepackage{amssymb,bm}
\usepackage{amsthm}
\usepackage{color}
\usepackage{amsmath}
\numberwithin{equation}{section}
\usepackage[]{amsfonts}
\usepackage{fancyhdr}
\usepackage[]{graphicx,wrapfig}
\usepackage{enumitem}
\usepackage{array}

\newtheorem{theorem}{Theorem}[section]

\newtheorem{proposition}{Proposition}[section]
\newtheorem{lemma}{Lemma}[section]

\newtheorem{remark}[theorem]{Remark}

\usepackage{graphicx}
\usepackage[dvipsnames]{xcolor}
\newtheorem{example}{Example}
\begin{document}
\title[Cyclic codes of length $n$ over finite chain rings ]{Cyclic codes of length $n$ over finite chain rings}
\author[S. Antil]{Seema Antil}
\email{seema.22maz0010@iitrpr.ac.in, seemaantil97@gmail.com}  
\author[G. Kaur]{Gurleen Kaur}
\email{gurleenkaur992gk@gmail.com}
\author[M. Khan]{Manju Khan}
\email{manju@iitrpr.ac.in}
\address{Department of Mathematics, Indian Institute of Technology, Ropar, Rupnagar-140001, Punjab, India.}
\begin{abstract}
\noindent In this paper, the cyclic codes of length $n$, where $n$ is odd with certain restrictions, over a finite chain ring $R$ of order $2^a$, have been studied using the structure of group algebra. The primitive idempotents of $RG$ of a finite cyclic group $G$ have been determined, and then the number of codewords and the minimum weight corresponding to a cyclic code have been computed.
\end{abstract}
\maketitle
\section{\textbf{INTRODUCTION}}\label{section1}
Let ${R}$ be a finite commutative chain ring with unity of order $2^a$ such that the order of the unique maximal ideal $M$ of $R$ is $2^{a-1}$, which is generated by a nilpotent element $s$, with nilpotency index $t$. Thus $\overline{R}=\frac{R}{M}$ is a field with two elements. Let ${R}^n$ denote the set of all $n$ tuples of the elements of ${R}$, which forms a ring with the natural way of addition and multiplication. Any subset $\mathcal{C}$ of ${R}^n$ is a $\mathbf{code}$ of length $n$. A code $\mathcal{C}$ is said to be a $\mathbf{linear}$ code if $\mathcal{C}$ is a submodule  of ${R}^n$ over ${R}$.
A linear code $\mathcal{C}$ is $\mathbf{cyclic}$ if for every $(c_0,c_1,...,c_{n-1})\in \mathcal{C}$ implies $(c_{n-1},c_0,...,c_{n-2})\in \mathcal{C}$. We consider the map $\eta: {R}^n \longrightarrow \frac{{R}[x]}{\langle x^n-1\rangle} $, $(c_0,c_1,...,c_{n-1})$ $\longmapsto$ $c_0+c_1x+...+c_{n-1}x^{n-1}$. Then $\eta$ is an ${R}$-module isomorphism. Therefore, there is a 1-1 correspondence between cyclic codes and ideals of $\frac {R[x]}{\langle x^n-1\rangle}$. If ${x^n-1 }=f_0f_1\cdots f_{m}$ is a product of irreducible polynomials over $R$, then $ \dfrac{R[x]}{\langle x^n-1\rangle}\cong  \mathop{\oplus}\limits_{i=0}^m\frac{R[x]}{\langle f_i \rangle }$.

Let $G$ be a cyclic group of odd order. Since ${R}G$ is semisimple, there exist  primitive (orthogonal) idempotents ${e_0}, \cdots ,{e_{m}} $ such that ${R}G={R}G{e_0} \oplus \cdots \oplus {R}G{e_{m}}$. Further, $RGe_i \cong \frac{R[x]}{\langle f_i\rangle }$ for all $0\leq i \leq m$. Since $\frac{R[x]}{\langle f_i\rangle }$ is a chain ring with  ideals of the form $s^j\frac{R[x]}{\langle f_i\rangle }$ for all $0\leq j \leq {t-1}$ (for more details, see \cite{HS}). Let $I$ be an ideal of  $\frac{{R}[x]}{\langle x^n-1\rangle}$. Then $I$ can be written as $\mathop{\oplus}\limits_{i=0}^mI_i$, where $I_i$ is an ideal of $RGe_i$, and hence $\mathcal{C}$ is of the form $\mathcal{C}=\mathop{\oplus}\limits_{i=0}^m\langle s^je_i \rangle $ for some $0 \leq j \leq t-1$.

 For an element $\alpha =\mathop{\sum}\limits_{g\in G}\alpha_gg\in RG$, the $\mathbf{support}$ of $\alpha$ is defined as $supp(\alpha)=\{g \in G ~|~ \alpha_g\neq 0\}$.
The $\mathbf{weight}$ of $\alpha$ is defined as $wt(\alpha)=|supp(\alpha)|.$
For an ideal $\mathcal{I}$ of $RG$, define the $\mathbf{minimum~ weight}$ of $\mathcal{I}$ as  $wt(I)=min\{wt(\alpha)~|~ 0\neq \alpha\in \mathcal{I}\}.$

The following results are useful for computing the number of codewords and the minimum weight of a cyclic code $\mathcal{C}$.

\begin{proposition} \rm{(\cite[Proposition $2.1$]{HS})}
    Let $R$ be a finite chain ring with maximal ideal $M = \langle s \rangle$. Then the following are equivalent:
\begin{itemize}
    \item[(i)] $R$ is a local ring and the maximal ideal $M$ of $R$ is principal.
    \item[(ii)] $R$ is a local principal ideal ring.
    \item[(iii)] $R$ is a chain ring.
\end{itemize}
\end{proposition}

\begin{proposition} \rm{(\cite[Proposition $2.2$]{HS})} 
 Let $R$ be a finite chain ring with maximal ideal ${M}=\langle s\rangle $. Then
\begin{itemize}
    \item[(a)]  for some prime $q$ and positive integers $\tau$ and $l$ with $\tau \geq l$, we have
    \[
    |{R}| = q^\tau, \quad |\overline{{R}}| = q^l,
    \]
    and the characteristic of $R$ and $\overline{ R}$ are powers of $q$.

    \item[(b)]  \[
    |\langle a^i \rangle| = |\overline{\mathcal{R}}|^{t - i}, \quad \text{for } 0 \leq i \leq t.
    \]
    In particular, 
    \[
    | R | = |\overline{{R}}|^l,
    \]so $\tau = lt$.
\end{itemize}

\end{proposition}

\begin{proposition} \rm{(\cite[Proposition $2.2$]{ACP})} \label{pq}
   Let $R$  be a chain ring, and ${M}$ be the maximal ideal of  $R$. Then,  for any group $G$
\[
\frac{{RG}} {MG} \cong \left (\frac{R}{M}\right){G}
\]
as ${RG}$-module.
\end{proposition}

\begin{theorem} \cite{ACP} \label{3}
    Let ${R}$ be a finite chain ring such that $|{R}| = q^k$, and let ${G}$ be a cyclic group of order $n$ such that $(q, n) = 1$.  
If $\{\overline{e}_0, \ldots, \overline{e}_m\}$ is the complete set of primitive orthogonal idempotents in $\overline{{R}}{G}$, then $\{e_0, \ldots, e_m\}$ is the complete set of primitive orthogonal idempotents in ${RG}$.

\end{theorem}

\begin{theorem} \cite{ACP}
Let ${M} = \langle s \rangle$ be the maximal ideal of a finite chain ring ${R}$ such that $\left |\frac{R}{M}\right | = q^l$. Let ${G}$ be the cyclic group of order $n$ such that $(q, n) = 1$.  If $\mathcal{I}$ is an ideal of ${RG} e_i$ and $t$ is the nilpotency index of $s$ then $\mathcal{I} = \langle s^{j} e_i \rangle, ~\text{where~ } 0 \leq j \leq t.$

\end{theorem}
Let $H$ be a finite subgroup of a group $G$ and $R$ be a ring such that the order of $H$ is invertible in $R$. Then define $\widehat{H}=\frac{1}{|H|}\mathop{\sum}\limits_{h\in H}h$, which is an idempotent of $RG$ if and only if $H$ is a normal subgroup of $G$. In particular, for $H=G$, $\widehat{G}$ is a primitive idempotent of $RG$. If $H=\langle a\rangle$ then we denote  $\widehat{H}=\widehat{\langle a\rangle}$ by $\widehat{a}$. To calculate the primitive idempotents of $RG$, the following theorem plays a crucial role. Further, the proof of the theorem also provides information on simple components. We restate the  Theorem $3.1$ of \cite{RFM} as follows:
\begin{theorem} \label{xy} 
    Let $K$ be a field with $q$ elements and ${A}$ be a cyclic group of order $p^n$ , where $p$ is an odd prime and $(q, p^n)=1$ such that order of $q$ is $\phi(p^n) $ in $U(\mathbb{Z}_{p^n})$. Let
\begin{align*}
{A} = A_0 \supseteq A_1 \supseteq \cdots \supseteq A_n = \{1\}
\end{align*}
be the descending chain of subgroups of ${A}$ with $A_i= \langle a^{p^i}\rangle,\; 1\leq i\leq n$.
Then, the primitive idempotents of ${KA}$ are given by
\[
e_0 = \widehat{A}_0 \text{ and } e_i = \widehat{A_{i-1}}-\widehat{A_i}, \quad 1 \leq i \leq n
\] and the corresponding simple components are as follows
$KGe_i\cong K_{p^i-p^{i-1}}.$
\end{theorem}

The study of codes over finite chain rings is an active research topic in the area of coding theory, and one of its applications is the study of error-correcting codes through the number of codewords and the minimum weight of a given code $\mathcal{C}$. A code $\mathcal{C}$ as an ideal of finite group algebra $RG$ was introduced by Bermann \cite{SD} and MacWilliams \cite{FJ}. For more details (see \cite{HS, SJ, ACP, ZH, MJ, GA}). The dimension and the minimum distance for a cyclic code $\mathcal{C}$ of length $n$ over finite field of order $q$, such that $(q,n)=1$ has been studied in \cite{ PS}. Chalom et al. \cite{GRMC} provided the dimension and the weight for the cyclic group of order $p^mq^n$ over finite field of two elements. Silva, in his Ph.D thesis \cite{AT}, studied cyclic codes, their dual and self-dual codes over finite chain rings by using the structure of group algebra. In \cite{ACP},  Silva et al. studied the cyclic codes of length $2p^n$ over a finite chain ring.

In this article, we generalize the work done by Silva \cite{ACP} for cyclic codes of length $n$, where $n$ is odd with certain restrictions, over a finite chain ring $R$ with $2^a$ elements using group algebra approach. The set of primitive idempotents of $RG$ plays an important role in computing the number of codewords and the minimum weight of a cyclic code. In Section \ref{5}, we compute the number of primitive idempotents of $RG$ and provide an explicit description of some of the primitive idempotents of $RG$, which has been used in the subsequent sections to determine the number of codewords and the minimum weight of ideals in $RG$. In Section \ref{6}, we compute the number of codewords and in Section \ref{7}, we discuss the minimum weight for a cyclic code $\mathcal{C}.$ 
\section{\textbf{THE NUMBER OF SIMPLE COMPONENTS OF \it{RG}}}\label{5}
Let $R$ be a finite semisimple chain ring of order $2^a$ and $G$ be a cyclic group of odd order. In this section, we compute primitive idempotents of the semisimple group ring $RG$. The main theorem of this section is as follows:
\begin{theorem} \label{4}
      Let $R$ be a local ring and $M=\langle s \rangle $ be the maximal ideal of $R$, with nilpotency index $t$ such that $\frac{R}{M}\cong \mathbb{F}_2.$ The following are primitive idempotents of $RG$

\begin{enumerate}
    \item          $e_{0\cdots 0}= {\widehat{a_1}\cdots \widehat{a_r}}.$
   
        \item $e_{0\cdots j_{i_1}\cdots0} ={\widehat{a_1}}\cdots ( {\widehat{a_{i_1}^{p_{i_1}^{j_{i_1}}}}}-{\widehat{a_{i_1}^{p_{i_1}^{j_{i_1}-1}}}} )\cdots {\widehat{a_r}}.$
        \item \begin{enumerate}
            \item  $e_{0\cdots j_{i_1}\cdots j_{i_2} \cdots 0}^{(1)}={\widehat{a_1}}\cdots (u_{i_1}u_{i_2}+u_{i_1}^2u_{i_2}^2)^{2^{t-1}}\cdots {\widehat{a_r}}=(e{'}_{0\cdots j_{i_1}\cdots j_{i_2} \cdots 0}^{(1)})^{2^{t-1}},$
        \item   $e_{0\cdots j_{i_1}\cdots j_{i_2} \cdots 0}^{(2)}={\widehat{a_1}}\cdots (u_{i_1}^2u_{i_2}+u_{i_1}u_{i_2}^2)^{2^{t-1}}\cdots{\widehat{a_r}}=(e{'}_{0\cdots j_{i_1}\cdots j_{i_2} \cdots 0}^{(2)})^{2^{t-1}},$
        
        \end{enumerate}
   where ~ $e_{0\cdots j_{i_1}\cdots j_{i_2} \cdots 0}^{(1)}$ and $e_{0\cdots j_{i_1}\cdots j_{i_2} \cdots 0}^{(1)}$ are obtained in a way  $$e{'}_{0\cdots j_{i_1}\cdots j_{i_2} \cdots 0}^{(1)}+e{'}_{0\cdots j_{i_1}\cdots j_{i_2} \cdots 0}^{(2)}={\widehat{a_1}}\cdots( {\widehat{a_{i_1}^{p_{i_1}^{j_{i_1}}}}}-{\widehat{a_{i_1}^{p_{i_1}^{j_{i_1}-1}}}} )\cdots( {\widehat{a_{i_2}^{p_{i_2}^{j_{i_2}}}}}-{\widehat{a_{i_2}^{p_{i_2}^{j_{i_2}-1}}}} )\cdots {\widehat{a_r}}=e_{0\cdots j_{i_1}\cdots j_{i_2} \cdots 0}.$$

   \item $e_{j_1\cdots j_r}^{(1)}=(u_1u_2\cdots u_r+u_1^2u_2^2\cdots u_r^2)^{2^{t-1}}=(e{'}_{j_1\cdots j_r}^{(1)})^{2^{t-1}}$
\end{enumerate}
 ~~~~ where   $e{'}_{j_1\cdots j_r}^{(1)}+\cdots+e{'}_{j_1\cdots j_r}^{(2^{r-1})}=( {\widehat{a_{1}^{p_{1}^{j_{1}}}}}-{\widehat{a_{1}^{p_{1}^{j_{1}-1}}}})\cdots( {\widehat{a_{r}^{p_{r}^{j_{r}}}}}-{\widehat{a_{r}^{p_{r}^{j_{r}-1}}}} )=e_{j_1\cdots j_r},$  and $u_{i}$ defined as
     
\[
u_{i}^{2^{t-1}} =
\begin{cases}
{\widehat{a_{i}^{p_i^{j_{i}}}}}(a_{i}^{{2^0}{p_{i}^{j_{i}-1}}} + a_{i}^{2^2p_{i}^{j_{i}-1}} + \cdots + a_{i}^{2^{{p_{i}}-3}p_{i}^{j_{i}-1}})^{2^{t-1}}, & \text{if } p_{i} \equiv 1 \pmod{4},~or \\
{\widehat{a_{i}^{p_i^{j_{i}}}}}(1+a_{i}^{{2^0}{p_{i}^{j_{i}-1}}} + a_{i}^{2^2p_{i}^{j_{i}-1}} + \cdots + a_{i}^{2^{{p_{i}}-3}p_{i}^{j_{i}-1}})^{2^{t-1}}, & \text{if } p_{i} \equiv 3 \pmod{4}.
\end{cases}
\]
 \end{theorem}

Since $\frac{R}{M}\cong \mathbb{F}_2$, using a similar approach as given in Theorem $4.10$ of \cite{GRMC}, we firstly compute primitive idempotents of $\mathbb{F}_2G$. 

\begin{theorem} \label{uvw}
    Let $G=<a_1> \times <a_2>\times \cdots \times <a_r>$  be a cyclic group of order $p^{n_1}_1p^{n_2}_2\cdots p^{n_r}_r$, where $o(a_i)=p_i^{n_i}$ $\forall~1\leq i \leq r$, $n_1,n_2,\cdots,n_r$  are natural numbers, and,  $p_1$, $p_2$, $\cdots$, $p_r$ are distinct odd primes, which satisfy the following conditions: 
  \begin{enumerate}
      \item $gcd(p_i{}-1,{p_{i{'}}}-1)=2, \text{ for }~ 1\leq i\neq i{'} \leq r$
      \item $ {2}  \text{ is a primitive element of } \mathbb{Z}_{p_i^2}, ~ \text{ for all }~1\leq i \leq r. $
  \end{enumerate}
     Then the number of primitive idempotents of $\mathbb{F}_2G$ are $$2^{r-1}n_1n_2\dots n_r+\cdots+2^2 \cdot\sum_{1\leq i_1 < i_2 < i_3 \leq n_r} n_{i_1}n_{i_2}n_{i_3}+2\cdot \sum_{1\leq i_1 < i_2 \leq n_r} n_{i_1}n_{i_2}+\sum_{i=1}^rn_i+1.$$ 
    \begin{proof}
 Since $2$ is a primitive element of $\mathbb{Z}_{p_i^2}$, by [\cite{RFM}, Theorem $3.1$], we have
\begin{align*}
    \mathbb{F}_2C_{p_i^{n_i}}& = \mathbb{F}_2C_{p_i^{n_i}}\cdot \widehat{a_i}\oplus \bigoplus_{j_i=1}^{n_i} \mathbb{F}_2C_{p_i^{n_i}}\cdot (\widehat{a_i^{p_i^{j_i}}}-\widehat{a_i^{p_i^{j_i-1}}})\\&
    \cong \mathbb{F}_2 \oplus \bigoplus_{j_i=1}^{n_i}  \mathbb{F}_{{2}^{(p_{i}^{j_{i}}-p_{i}^{j_{i}-1})}}, \text{for} ~1\leq i \leq r.
\end{align*}
Since ${2}$ is a primitive element of $\mathbb{Z}_{p_i^{n_i}}$ and by [\cite{GRMC}, Lemma II.1], we have the following decomposition:
\begin{align*}
    \mathbb{F}_2G&=\mathbb{F}_2(C_{p_1^{n_1}}\times C_{p_2^{n_2}}\times \cdots \times  C_{p_r^{n_r}})=\mathbb{F}_2C_{p_1^{n_1}}\otimes \mathbb{F}_2 C_{p_2^{n_2}}\otimes \cdots \otimes \mathbb{F}_2 C_{p_r^{n_r}}\\
    & \cong (\mathbb{F}_2\oplus \bigoplus_{j_1=1}^{n_1} \mathbb{F}_{2^{(p_1^{j_1}-p_1^{j_1-1})}})\otimes (\mathbb{F}_2\oplus  \bigoplus_{j_2=1}^{n_2} \mathbb{F}_{2^{(p_2^{j_2}-p_2^{j_2-1})}}) \otimes \cdots \otimes (\mathbb{F}_2\oplus  \bigoplus_{j_r=1}^{n_r} \mathbb{F}_{2^{(p_r^{j_r}-p_r^{j_r-1})}})\\&
    =\mathbb{F}_2\oplus \bigoplus_{i_1=1}^r \bigoplus_{j_{i_1}=1}^{n_{i_1}} \mathbb{F}_{2^{(p_{i_1}^{j_{i_1}}-p_i^{j_{i_1}-1})}} \oplus  2\cdot \bigoplus_{1= i_1 < i_2}^ r \bigoplus_{j_{i_1}=1}^{n_{i_1}} \bigoplus_{j_{i_2}=1}^{n_{i_2}}   \mathbb{F}_{2^{\frac{(p_{i_1}^{j_{i_1}}-p_{i_1}^{j_{i_1}-1})(p_{i_2}^{j_{i_2}}-p_{i_2}^{j_{i_2}-1})}{2}}} \\&\quad {2^2}\cdot \oplus \bigoplus_{1= i_1 < i_2 <i_3}^ r\bigoplus_{j_{i_1}=1}^{n_{i_1}}\bigoplus_{j_{i_2}=1}^{n_{i_2}}\bigoplus_{j_{i_3}=1}^{n_{i_3}}   \mathbb{F}_{2^{\frac{(p_{i_1}^{j_{i_1}}-p_{i_1}^{j_{i_1}-1})(p_{i_2}^{j_{i_2}}-p_{i_2}^{j_{i_2}-1})(p_{i_3}^{j_{i_3}}-p_{i_3}^{j_{i_3}-1})}{4}}} \oplus \cdots  \\&\quad  {2^{r-1}}\cdot \oplus  \bigoplus_{j_{i_1}=1}^{n_{i_1}} \cdots \bigoplus_{j_{i_r}=1}^{n_{i_r}} \mathbb{F}_{2^{\frac{(p_1^{j_1}-p_1^{j_1-1})(p_2^{j_2}-p_2^{j_2-1})\cdots (p_r^{j_r}-p_r^{j_r-1})}{2^{r-1}}}},
    \end{align*}  where $n\cdot \mathbb{F}$ denotes the direct sum of $n$ copies of $\mathbb{F}$.
    Therefore, there exist $$2^{r-1}n_1n_2\dots n_r+\cdots+2^2\cdot \sum_{1\leq i_1 < i_2 < i_3 \leq n_r} n_{i_1}n_{i_2}n_{i_3}+2\cdot \sum_{1\leq i_1 < i_2 \leq n_r} n_{i_1}n_{i_2}+\sum_{i=1}^rn_i+1$$  simple components in the above decomposition.
    Since the number of simple components is equal to the number of primitive idempotents (see Proposition 3.6.3, \cite{CSK}), this completes the proof.\qed 
    \vspace{.25cm}\\
   We explicitly compute some of the primitive idempotents and their corresponding simple component, which will be required in the next section. For the sake of convenience, we define $$\overline{e}_{0\cdots{{j_{i_1}}}\cdots{{j_{i_2}}} \cdots{{j_{i_k}}}\cdots 0}=\widehat{\overline{a}_1}\cdots\widehat{\overline{a}_{{i_1}-1}}(\widehat{\overline{a}_{i_1}^{p^{j_{i_1}}_{i_1}}}-\widehat{\overline{a}_{i_1}^{p^{j_{i_1}-1}_{i_1}}})\widehat{\overline{a}_{{i_1}+1}}\cdots(\widehat{\overline{a}_{i_2}^{p^{j_{i_2}}_{i_2}}}-\widehat{\overline{a}_{i_2}^{p^{j_{i_2}-1}_{i_2}}})\cdots (\widehat{\overline{a}_{i_k}^{p^{j_{i_k}}_{i_k}}}-\widehat{\overline{a}_{i_l}^{p^{j_{i_l}-1}_{i_l}}})\cdots\widehat{\overline{a}_r},$$ for $1\leq j_{i_l}\leq {n_{i_l}}$ for all $1\leq i_l \leq r$. \\
    One can observe that $\overline{e}_{0\cdots{{j_{i_1}}}\cdots{{j_{i_2}}} \cdots{{j_{i_k}}}\cdots 0}$ is an idempotent of $\mathbb{F}_2G.$
    Now, we construct the primitive idempotents of $\mathbb{F}_{2}G$. Firstly, consider $\overline{e}_{0\cdots 0}=\widehat{\overline{a}_1}~\widehat{\overline{a}_2}\cdots \widehat{\overline{a}_r}$. Since we are in the field $\mathbb{F}_2$, $\widehat{\overline{a}_1}~\widehat{\overline{a}_2}\cdots \widehat{\overline{a}_r}=\widehat{G}$ is clearly an idempotent of $\mathbb{F}_2G.$. Further,
    \begin{align*}
         \mathbb{F}_2G\cdot \widehat{G}\cong \mathbb{F}_2(G/G)\cong \mathbb{F}_2
    \end{align*}
        Hence, $\overline{e}_{0\cdots0}$ is a primitive idempotent of $\mathbb{F}_2 G$.\\
        Next, we consider the element $\overline{e}_{0\cdots{{j_{i_1}}}\cdots0}$. 
        \begin{align*}
            \mathbb{F}_2G\overline{e}_{0\cdots{{j_{i_1}}}\cdots0}\cong\mathbb{F}_2 C_{p_{i_1}^{n_{i_1}}}\cdot (\widehat{\overline{a}_{i_1}^{p_{i_1}^{j_{i_1}}}}-\widehat{\overline{a}_{i_1}^{p_{i_1}^{j_{i_1}-1}}})\cong \mathbb{F}_{2^{p_{i_1}^{j_{i_1}}-p_{i_1}^{j_{i_1}-1}}},
        \end{align*}
       using Theorem \ref{xy}\ . Therefore, $\overline{e}_{0\cdots{{j_{i_1}}}\cdots0}$ is a primitive idempotent of $\mathbb{F}_2G$.\\
       For the idempotent $\overline{e}_{0\cdots{{j_{i_1}}}\cdots{{j_{i_2}}}\cdots 0}$, we proceed along similar lines as in Theorem 4.10 of \cite{GRMC} and observe the corresponding component 
       \begin{align*}
         \mathbb{F}_2 G\overline{e}_{0\cdots{{j_{i_1}}}\cdots{{j_{i_2}}}\cdots 0}&\cong \mathbb{F}_2 G \overline{e}_{0\cdots{{j_{i_1}}}\cdots{{j_{i_2}}}\cdots 0}^{(1)} \oplus \mathbb{F}_2 G \overline{e}_{0\cdots{{j_{i_1}}}\cdots{{j_{i_2}}}\cdots 0}^{(2)},\\&\cong 2\cdot  \mathbb{F}_{2^{\frac{(p_{i_1}^{j_{i_1}}-p_{i_1}^{j_{i_1}-1})(p_{i_2}^{j_{i_2}}-p_{i_2}^{j_{i_2}-1})}{2}}},
    \end{align*}
    where
    \begin{align*}
        \overline{e}_{0\cdots{{j_{i_1}}}\cdots{{j_{i_2}}}\cdots 0}^{(1)} =&\widehat{\overline{a}_1}\cdots(\widehat{\overline{a}_{i_1}^{p^{j_{i_1}}_{i_1}}}-\widehat{\overline{a}_{i_1}^{p^{j_{i_1}-1}_{i_1}}})\cdots(\widehat{\overline{a}_{i_2}^{p^{j_{i_2}}_{i_2}}}-\widehat{\overline{a}_{i_2}^{p^{j_{i_2}-1}_{i_2}}})\cdots \widehat{\overline{a}_r}\\&+\widehat{\overline{a}_1}\widehat{\overline{a}_2}\cdots \widehat{\overline{a}_{i_1-1}}\widehat{\overline{a}_{{i_1}+1}}\cdots \cdots \widehat{\overline{a}_{i_2-1}}\widehat{\overline{a}_{i_2+1}} \cdots \widehat{\overline{a}_r}(\overline{u}_{i_1}\overline{u}_{i_2}+\overline{u}_{i_1}^2\overline{u}_{i_2}^2)
    \end{align*}
    and 
    \begin{align*}
        \overline{e}_{0\cdots{{j_{i_1}}}\cdots{{j_{i_2}}}\cdots 0}^{(2)} =&\widehat{\overline{a}_1}\cdots(\widehat{\overline{a}_{i_1}^{p^{j_{i_1}}_{i_1}}}-\widehat{\overline{a}_{i_1}^{p^{j_{i_1}-1}_{i_1}}})\cdots(\widehat{\overline{a}_{i_2}^{p^{j_{i_2}}_{i_2}}}-\widehat{\overline{a}_{i_2}^{p^{j_{i_2}-1}_{i_2}}})\cdots \widehat{\overline{a}_r}\\&+\widehat{\overline{a}_1}\widehat{\overline{a}_2}\cdots \widehat{\overline{a}_{i_1-1}}\widehat{\overline{a}_{{i_1}+1}}\cdots \cdots \widehat{\overline{a}_{i_2-1}}\widehat{\overline{a}_{i_2+1}} \cdots \widehat{\overline{a}_r}(\overline{u}_{i_1}^2\overline{u}_{i_2}+\overline{u}_{i_1}\overline{u}_{i_2}^2)
    \end{align*}
    with $\overline{u}_{i_1},  \overline{u}_{i_1}^2, $ being defined as 
\[
\overline{u}_{i_1} =
\begin{cases}
\widehat{\overline{a}_{i_1}^{p_{i_1}^{j_{i_1}}}}(\overline{a}_{i_1}^{{2^0}{p_{i_1}^{j_{i_1}-1}}} + \overline{a}_{i_1}^{2^2p_{i_1}^{j_{i_1}-1}} + \cdots + \overline{a}_{i_1}^{2^{{p_{i_1}}-3}p_{i_1}^{j_{i_1}-1}}), & \text{if } p_{i_1} \equiv 1 \pmod{4},~or \\
\widehat{\overline{a}_{i_1}^{p_{i_1}^{j_{i_1}}}}(1+\overline{a}_{i_1}^{{2^0}{p_{i_1}^{j_{i_1}-1}}} + \overline{a}_{i_1}^{2^2p_{i_1}^{j_{i_1}-1}} + \cdots + \overline{a}_{i_1}^{2^{{p_{i_1}}-3}p_{i_1}^{j_{i_1}-1}}), & \text{if } p_{i_1} \equiv 3 \pmod{4},
\end{cases}
\]

\[
\overline{u}_{i_1}^{2} =
\begin{cases}
\widehat{\overline{a}_{i_1}^{p_{i_1}^{j_{i_1}}}}(\overline{a}_{i_1}^{{2}{p_{i_1}^{j_{i_1}-1}}} + \overline{a}_{i_1}^{2^3p_{i_1}^{j_{i_1}-1}} + \cdots + \overline{a}_{i_1}^{2^{{p_{i_1}}-2}p_{i_1}^{j_{i_1}-1}}), & \text{if } p_{i_1} \equiv 1 \pmod{4},~or \\
\widehat{\overline{a}_{i_1}^{p_{i_1}^{j_{i_1}}}}(1+\overline{a}_{i_1}^{{2}{p_{i_1}^{j_{i_1}-1}}} + \overline{a}_{i_1}^{2^3p_{i_1}^{j_{i_1}-1}} + \cdots + \overline{a}_{i_1}^{2^{{p_{i_1}}-2}p_{i_1}^{j_{i_1}-1}}), & \text{if } p_{i_1} \equiv 1 \pmod{4}
\end{cases}
\]
Similarly, $\overline{u}_{i_2}$ and $\overline{u}_{i_2}^2$ are defined as above by replacing $i_1$ with $i_2$.\\
 For the idempotent $\overline{e}_{0\cdots{{j_{i_1}}}\cdots{{j_{i_2}}} \cdots{{j_{i_3}}}\cdots 0}$ and for $i_1,i_2,i_3\in \{ 1,2,\cdots,r\}$, $j_{i_1},j_{i_2},j_{i_3}$ runs from $1$ to $n_{i_1}$,$1$ to $n_{i_2}$, $1$ to $n_{i_3}$ and all are pairwise coprime, we observe that
    \begin{align*}
    \mathbb{F}_2G \overline{e}_{0\cdots{{j_{i_1}}}\cdots{{j_{i_2}}} \cdots{{j_{i_3}}}\cdots 0} \cong& \mathbb{F}_2G \overline{e}_{0\cdots{{j_{i_1}}}\cdots{{j_{i_2}}} \cdots{{j_{i_3}}}\cdots 0}^{(1)} \oplus \mathbb{F}_2G \overline{e}_{0\cdots{{j_{i_1}}}\cdots{{j_{i_2}}} \cdots{{j_{i_3}}}\cdots 0}^{(2)}\\& \oplus \mathbb{F}_2G  \overline{e}_{0\cdots{{j_{i_1}}}\cdots{{j_{i_2}}} \cdots{{j_{i_3}}}\cdots 0}^{(3)}
    \oplus \mathbb{F}_2G \overline{e}_{0\cdots{{j_{i_1}}}\cdots{{j_{i_2}}} \cdots{{j_{i_3}}}\cdots 0}^{(4)}\\ \cong& 2^2 \cdot \mathbb{F}_{2^{\frac{(p_{i_1}^{j_{i_1}}-p_{i_1}^{j_{i_1}-1})(p_{i_2}^{j_{i_2}}-p_{i_2}^{j_{i_2}-1})(p_{i_3}^{j_{i_3}}-p_{i_3}^{j_{i_3}-1})}{4}}}
    \end{align*}
    where \begin{align*}
      \overline{e}_{0\cdots{{j_{i_1}}}\cdots{{j_{i_2}}} \cdots{{j_{i_3}}}\cdots 0}^{(1)} = &(\widehat{\overline{a}_1}\cdots(\widehat{\overline{a}_{i_1}^{p^{j_{i_1}}_{i_1}}}-\widehat{\overline{a}_{i_1}^{p^{j_{i_1}-1}_{i_1}}})\cdots(\widehat{\overline{a}_{i_2}^{p^{j_{i_2}}_{i_2}}}-\widehat{\overline{a}_{i_2}^{p^{j_{i_2}-1}_{i_2}}}) \cdots(\widehat{\overline{a}_{i_3}^{p^{j_{i_3}}_{i_3}}}-\widehat{\overline{a}_{i_3}^{p^{j_{i_3}-1}_{i_3}}}) \cdots\widehat{\overline{a}_r})+\\&\overline{u}_{i_1}\overline{u}_{i_2}\overline{u}_{i_3}+\overline{u}_{i_1}^2\overline{u}_{i_2}^2\overline{u}_{i_3}^2,\\
      \overline{e}_{0\cdots{{j_{i_1}}}\cdots{{j_{i_2}}} \cdots{{j_{i_3}}}\cdots 0}^{(2)} = &(\widehat{\overline{a}_1}\cdots(\widehat{\overline{a}_{i_1}^{p^{j_{i_1}}_{i_1}}}-\widehat{\overline{a}_{i_1}^{p^{j_{i_1}-1}_{i_1}}})\cdots(\widehat{\overline{a}_{i_2}^{p^{j_{i_2}}_{i_2}}}-\widehat{\overline{a}_{i_2}^{p^{j_{i_2}-1}_{i_2}}}) \cdots(\widehat{\overline{a}_{i_3}^{p^{j_{i_3}}_{i_3}}}-\widehat{\overline{a}_{i_3}^{p^{j_{i_3}-1}_{i_3}}}) \cdots\widehat{\overline{a}_r})+\\&\overline{u}_{i_1}\overline{u}_{i_2}^2\overline{u}_{i_3}+\overline{u}_{i_1}^2\overline{u}_{i_2}\overline{u}_{i_3}^2,\\
       \overline{e}_{0\cdots{{j_{i_1}}}\cdots{{j_{i_2}}} \cdots{{j_{i_3}}}\cdots 0}^{(3)} = &(\widehat{\overline{a}_1}\cdots(\widehat{\overline{a}_{i_1}^{p^{j_{i_1}}_{i_1}}}-\widehat{\overline{a}_{i_1}^{p^{j_{i_1}-1}_{i_1}}})\cdots(\widehat{\overline{a}_{i_2}^{p^{j_{i_2}}_{i_2}}}-\widehat{\overline{a}_{i_2}^{p^{j_{i_2}-1}_{i_2}}}) \cdots(\widehat{\overline{a}_{i_3}^{p^{j_{i_3}}_{i_3}}}-\widehat{\overline{a}_{i_3}^{p^{j_{i_3}-1}_{i_3}}}) \cdots\widehat{\overline{a}_r})+\\& \overline{u}_{i_1}\overline{u}_{i_2}\overline{u}_{i_3}^2+\overline{u}_{i_1}^2\overline{u}_{i_2}^2\overline{u}_{i_3},\\
         \overline{e}_{0\cdots{{j_{i_1}}}\cdots{{j_{i_2}}} \cdots{{j_{i_3}}}\cdots 0}^{(4)} = & (\widehat{\overline{a}_1}\cdots(\widehat{\overline{a}_{i_1}^{p^{j_{i_1}}_{i_1}}}-\widehat{\overline{a}_{i_1}^{p^{j_{i_1}-1}_{i_1}}})\cdots(\widehat{\overline{a}_{i_2}^{p^{j_{i_2}}_{i_2}}}-\widehat{\overline{a}_{i_2}^{p^{j_{i_2}-1}_{i_2}}}) \cdots(\widehat{\overline{a}_{i_3}^{p^{j_{i_3}}_{i_3}}}-\widehat{\overline{a}_{i_3}^{p^{j_{i_3}-1}_{i_3}}}) \cdots\widehat{\overline{a}_r})+\\& \overline{u}_{i_1}^2\overline{u}_{i_2}\overline{u}_{i_3}+\overline{u}_{i_1}\overline{u}_{i_2}^2\overline{u}_{i_3}^2,
    \end{align*}
    where $\overline{u}_{i_1}, \overline{u}_{i_1}^2, \overline{u}_{i_2}, \overline{u}_{i_2}^2, \overline{u}_{i_3}$ and $\overline{u}_{i_3}^2$ can be similarly obtained as we defined for $\overline{e}_{0\cdots j_{i_1}\cdots j_{i_2}\cdots 0}.$\\
    As $\overline{e}_{j_1\cdots j_r}$ can be written as sum of $2^{r-1}$ primitive idempotents, and we can compute some of the primitive idempotents which are sufficient to prove our main theorems in sections $3$ and $4$.
    Along similar lines, we look at the simple components of 
    \begin{align*}
        \mathbb{F}_2C_{p_1^{n_1}}(\widehat{\overline{a}_{1}^{p^{j_{1}}_{1}}}-\widehat{\overline{a}_{1}^{p^{j_{1}-1}_{1}}})\otimes \mathbb{F}_2C_{p_2^{n_2}}(\widehat{\overline{a}_{2}^{p^{j_{2}}_{2}}}-\widehat{\overline{a}_{2}^{p^{j_{2}-1}_{2}}}) \otimes \cdots \otimes \mathbb{F}_2C_{p_r^{n_r}}(\widehat{\overline{a}_{r}^{p^{j_r}_r}}-\widehat{\overline{a}_{r}^{p^{j_r-1}_r}}).
    \end{align*}
    Let $0\neq \overline{u}_i\in \mathbb{F}_2C_{p_i^{n_i}}(\widehat{\overline{a}_{i}^{p^{j_{i}}_{i}}}-\widehat{\overline{a}_{i}^{p^{j_{i}-1}_{i}}})$ be an element such that $\overline{u}_i^3=(\widehat{\overline{a}_{i}^{p^{j_{i}}_{i}}}-\widehat{\overline{a}_{i}^{p^{j_{i}-1}_{i}}})$ and $\overline{u}_i\neq (\widehat{\overline{a}_{i}^{p^{j_{i}}_{i}}}-\widehat{\overline{a}_{i}^{p^{j_{i}-1}_{i}}}) $, where $1\leq i \leq r$ (for more details, see \cite{GRMC}).\par 
    In $\mathbb{F}_2C_{p_{i_1}^{n_{i_1}}}(\widehat{\overline{a}_{i_1}^{p^{{j_{i_1}}}_{i_1}}}-\widehat{\overline{a}_{i_1}^{p^{{j_{i_1}-1}}_{i_1}}})\otimes ~\mathbb{F}_2C_{p_{i_2}^{n_{i_2}}}(\widehat{\overline{a}_{i_2}^{p^{{j_{i_2}}}_{i_2}}}-\widehat{\overline{a}_{i_2}^{p^{{j_{i_2}-1}}_{i_2}}})$, we have $(\overline{u}_{i_1}\overline{u}_{i_2})^3=(\widehat{\overline{a}_{i_1}^{p^{{j_{i_1}}}_{i_1}}}-\widehat{\overline{a}_{i_1}^{p^{{j_{i_1}-1}}_{i_1}}})(\widehat{\overline{a}_{i_2}^{p^{{j_{i_2}}}_{i_2}}}-\widehat{\overline{a}_{i_2}^{p^{{j_{i_2}-1}}_{i_2}}})$. Let $\overline{\alpha}=\overline{u}_{i_1}\overline{u}_{i_2}(\overline{u}_{i_1}\overline{u}_{i_2}+\overline{u}_{i_1}^2\overline{u}_{i_2}^2)=\overline{u}_{i_1}\overline{u}_{i_2}+\overline{u}_{i_1}^2\overline{u}_{i_2}+\overline{u}_{i_1}\overline{u}_{i_2}^2$. Similarly, $(\overline{u}_{i_1}^2\overline{u}_{i_2}^2)^3=(\widehat{\overline{a}_{i_1}^{p^{{j_{i_1}}}_{i_1}}}-\widehat{\overline{a}_{i_1}^{p^{{j_{i_1}-1}}_{i_1}}})(\widehat{\overline{a}_{i_2}^{p^{{j_{i_2}}}_{i_2}}}-\widehat{\overline{a}_{i_2}^{p^{{j_{i_2}-1}}_{i_2}}})$ and $\overline{\alpha}^2=\overline{u}_{i_1}^2\overline{u}_{i_2}^2(\overline{u}_{i_1}\overline{u}_{i_2}+\overline{u}_{i_1}^2\overline{u}_{i_2}^2)=\overline{u}_{i_1}\overline{u}_{i_2}^2+\overline{u}_{i_1}^2\overline{u}_{i_2}^2+\overline{u}_{i_1}\overline{u}_{i_2}$.
    Define 
    \begin{align*}A_{i_1} & =\overline{\alpha} ~\overline{w}+\overline{\alpha}^2~\overline{ w}^2=(\overline{u}_{i_1}\overline{u}_{i_2}+\overline{u}_{i_1}^2 \overline{u}_{i_2}+\overline{u}_{i_1}\overline{u}_{i_2}^2)^2(\overline{u}_1 \overline{u}_2\cdots {\overline{u}_{i_1-1}}{\overline{u}_{{i_1}+1}}\cdots {\overline{u}_{i_2-1}}{\overline{u}_{i_2+1}}\cdots \overline{u}_r)\\&+(\overline{u}_{i_1}\overline{u}_{i_2}^2+\overline{u}_{i_1}^2 \overline{u}_{i_2}^2+ \overline{u}_{i_1} \overline{u}_{i_2})(\overline{u}_1 \overline{u}_2\cdots {\overline{u}_{i_1-1}}{\overline{u}_{{i_1}+1}}\cdots \cdots {\overline{u}_{i_2-1}}{\overline{u}_{i_2+1}}\cdots \overline{u}_r)^2\\&=(\widehat{\overline{a}_1^{p_1^{j_1}}}-\widehat{\overline{a}_1^{p_1^{j_1-1}}})(\widehat{\overline{a}_2^{p_2^{j_2}}}-\widehat{\overline{a}_2^{p_2^{j_2-1}}})\cdots (\widehat{\overline{a}_r^{p_r^{j_r}}}-\widehat{\overline{a}_r^{p_r^{j_r-1}}})+(\overline{u}_1 \overline{u}_2\cdots {\overline{u}_{i_1-1}}\overline{u}_{i_1}^2{\overline{u}_{{i_1}+1}}\cdots  {\overline{u}_{i_2-1}}\overline{u}_{i_2}^2{u_{i_2+1}}\cdots \overline{u}_r)\\&+ (\overline{u}_1^2 \overline{u}_2^2\cdots {\overline{u}^2_{i_1-1}}\overline{u}_{i_1}{\overline{u}^2_{{i_1}+1}}\cdots  {{\overline{u}^2_{i_2-1}}} {\overline{u}_{i_2}}{\overline{u}^2_{i_2+1}}\cdots {\overline{u}_r^2}).
    \end{align*}
    \begin{align*}
         A_{i_2}&=\overline{\alpha}^2~\overline{w}+\overline{\alpha}~ \overline{w}^2\\&=(\overline{u}_{i_1}\overline{u}_{i_2}^2+\overline{u}_{i_1}^2 \overline{u}_{i_2}^2+\overline{u}_{i_1}\overline{u}_{i_2})(\overline{u}_1 \overline{u}_2\cdots {\overline{u}_{i_1-1}}{\overline{u}_{{i_1}+1}} \cdots {\overline{u}_{i_2-1}}{\overline{u}_{i_2+1}}\cdots \overline{u}_r)\\& + (\overline{u}_{i_1}\overline{u}_{i_2}+\overline{u}_{i_1}^2 \overline{u}_{i_2}+\overline{u}_{i_1}\overline{u}_{i_2}^2)(\overline{u}_1 \overline{u}_2\cdots {\overline{u}_{i_1-1}}{\overline{u}_{{i_1}+1}}\cdots {\overline{u}_{i_2-1}}{\overline{u}_{i_2+1}}\cdots \overline{u}_r)^2\\&= (\widehat{\overline{a}_1^{p_1^{j_1}}}-\widehat{\overline{a}_1^{p_1^{j_1-1}}})(\widehat{\overline{a}_2^{p_2^{j_2}}}-\widehat{\overline{a}_2^{p_2^{j_2-1}}})\cdots (\widehat{\overline{a}_r^{p_r^{j_r}}}-\widehat{\overline{a}_r^{p_r^{j_r-1}}})+(\overline{u}_1^2 \overline{u}_2^2\cdots {\overline{u}^2_{i_1-1}}\overline{u}_{i_1}^2{\overline{u}^2_{{i_1}+1}}\cdots  {{\overline{u}^2_{i_2-1}}} \overline{u}_{i_2}^2{\overline{u}^2_{i_2+1}}\cdots \overline{u}_r^2)\\&+ (\overline{u}_1 \overline{u}_2\cdots {\overline{u}_{i_1-1}}\overline{u}_{i_1}{\overline{u}_{{i_1}+1}}\cdots  {{\overline{u}_{i_2-1}}} \overline{u}_{i_2}{\overline{u}_{i_2+1}}\cdots \overline{u}_r).
    \end{align*} where $\overline{w}=\overline{u}_1 \overline{u}_2\cdots {\overline{u}_{i_1-1}}\overline{u}_{i_1}{\overline{u}_{{i_1}+1}}\cdots  {{\overline{u}_{i_2-1}}} \overline{u}_{i_2}{\overline{u}_{i_2+1}}\cdots \overline{u}_r$.\\
    First, we show that $A_{i_1}$ and $A_{i_2}$ are primitive idempotents of $$\mathbb{F}_2[C_{p_{i_1}^{n_{i_1}}}\times C_{p_{i_2}^{n_{i_2}}}](\overline{u}_{i_1}\overline{u}_{i_2}+\overline{u}_{i_1}^2\overline{u}_{i_2}^2) \otimes \mathbb{F}_2C_{p_1^{n_1}} (\widehat{\overline{a}_{1}^{p^{{j_{1}}}_{1}}}-\widehat{\overline{a}_{1}^{p^{{j_{1}-1}}_{1}}})\otimes \cdots \otimes \mathbb{F}_2C_{p_r^{n_r}} (\widehat{\overline{a}_{r}^{p^{{j_{r}}}_{r}}}-\widehat{\overline{a}_{r}^{p^{{j_{r}-1}}_{r}}}) .$$\\
    Observe that \begin{align*}
        A_{i_1}^2&=(\overline{\alpha}~ \overline{w}+\overline{\alpha}^2\overline{w}^2)(\overline{\alpha}~\overline{ w}+\overline{\alpha}^2\overline{w}^2)\\&
        =\overline{\alpha}^2\overline{w}^2+\overline{\alpha}^3\overline{w}^3+\overline{\alpha}^3\overline{w}^3+\overline{\alpha}^4\overline{w}^4=\overline{\alpha}~\overline{ w}+\overline{\alpha}^2\overline{w}^2=A_{i_1}.
    \end{align*} 
    Further, 
\begin{align*}
  \mathbb{F}_2G\overline{e}_{j_1\cdots j_r}&\cong \bigoplus_{i=1}^{2^{r-1}} \mathbb{F}_2G\overline{e}_{j_1\cdots j_r}^{(i)} \cong 2^{r-1} \cdot \mathbb{F}_{2^{\frac{(p_1^{j_1}-p_1^{j_1-1})\cdots (p_r^{j_r}-p_r^{j_r-1})}{2^{r-1}}}},
\end{align*}
which implies that $A_{i_1}$ and $A_{i_2}$ are primitive idempotents.
Define \begin{align*}B_{i_1}&=\overline{\beta}~ \overline{w}+\overline{\beta}^2\overline{w}^2\\&=(\overline{u}_{i_1}\overline{u}_{i_2}+\overline{u}_{i_1}^2\overline{u}_{i_2}+\overline{u}_{i_1}\overline{u}_{i_2}^2)^2(\overline{u}_1 \overline{u}_2\cdots {\overline{u}_{i_1-1}}{\overline{u}_{{i_1}+1}}\cdots {\overline{u}_{i_2-1}}{\overline{u}_{i_2+1}}\cdots \overline{u}_r)\\&+(\overline{u}_{i_1}\overline{u}_{i_2}^2+\overline{u}_{i_1}^2\overline{u}_{i_2}^2+\overline{u}_{i_1}\overline{u}_{i_2})(\overline{u}_1 \overline{u}_2\cdots {\overline{u}_{i_1-1}}{\overline{u}_{{i_1}+1}}\cdots \cdots {\overline{u}_{i_2-1}}{\overline{u}_{i_2+1}}\cdots \overline{u}_r)^2\\&=(\widehat{\overline{a}_1^{p_1^{j_1}}}-\widehat{\overline{a}_1^{p_1^{j_1-1}}})(\widehat{\overline{a}_2^{p_2^{j_2}}}-\widehat{\overline{a}_2^{p_2^{j_2-1}}})\cdots (\widehat{\overline{a}_r^{p_r^{j_r}}}-\widehat{\overline{a}_r^{p_r^{j_r-1}}})+(\overline{u}_1 \overline{u}_2\cdots {\overline{u}_{i_1-1}}\overline{u}_{i_1}^2{\overline{u}_{{i_1}+1}}\cdots  {\overline{u}_{i_2-1}}\overline{u}_{i_2}{\overline{u}_{i_2+1}}\cdots \overline{u}_r)\\&+ (\overline{u}_1^2 \overline{u}_2^2\cdots {\overline{u}^2_{i_1-1}}\overline{u}_{i_1}{\overline{u}^2_{{i_1}+1}}\cdots  {{\overline{u}^2_{i_2-1}}} \overline{u}_{i_2}^2{\overline{u}^2_{i_2+1}}\cdots \overline{u}_r^2).\\
         B_{i_2}&=\overline{\beta}^2 \overline{w}+\overline{\beta} \overline{w}^2\\&=(\overline{u}_{i_1}\overline{u}_{i_2}^2+\overline{u}_{i_1}^2\overline{u}_{i_2}^2+\overline{u}_{i_1}\overline{u}_{i_2})(\overline{u}_1 \overline{u}_2\cdots {\overline{u}_{i_1-1}}{\overline{u}_{{i_1}+1}} \cdots {\overline{u}_{i_2-1}}{\overline{u}_{i_2+1}}\cdots \overline{u}_r)\\& + (\overline{u}_{i_1}\overline{u}_{i_2}+\overline{u}_{i_1}^2\overline{u}_{i_2}+\overline{u}_{i_1}\overline{u}_{i_2}^2)(\overline{u}_1 \overline{u}_2\cdots {\overline{u}_{i_1-1}}{\overline{u}_{{i_1}+1}}\cdots {\overline{u}_{i_2-1}}{\overline{u}_{i_2+1}}\cdots \overline{u}_r)^2\\&= (\widehat{\overline{a}_1^{p_1^{j_1}}}-\widehat{\overline{a}_1^{p_1^{j_1-1}}})(\widehat{\overline{a}_2^{p_2^{j_2}}}-\widehat{\overline{a}_2^{p_2^{j_2-1}}})\cdots (\widehat{\overline{a}_r^{p_r^{j_r}}}-\widehat{\overline{a}_r^{p_r^{j_r-1}}})+(\overline{u}_1^2 \overline{u}_2^2\cdots {\overline{u}^2_{i_1-1}}\overline{u}_{i_1}^2{\overline{u}^2_{{i_1}+1}}\cdots  {{\overline{u}^2_{i_2-1}}} \overline{u}_{i_2}{\overline{u}^2_{i_2+1}}\cdots \overline{u}_r^2)\\&+ (\overline{u}_1 \overline{u}_2\cdots {\overline{u}_{i_1-1}}\overline{u}_{i_1}{\overline{u}_{{i_1}+1}}\cdots  {{\overline{u}_{i_2-1}}} \overline{u}_{i_2}^2{\overline{u}_{i_2+1}}\cdots \overline{u}_r).
    \end{align*}
Using similar process, we can show that $B_{i_1}=\overline{\beta} \overline{w}+\overline{\beta}^2 \overline{w}^2$ and $B_{i_2}=\overline{\beta}^2 \overline{w}+\overline{\beta} \overline{w}^2,$ where $\beta=\overline{u}_{i_1}\overline{u}_{i_2}^2+\overline{u}_{i_1}^2\overline{u}_{i_2}^2+\overline{u}_{i_1}\overline{u}_{i_2},$ are primitive idempotents of $\mathbb{F}_2[C_{p_{i_1}^{n_{i_1}}}\times C_{p_{i_2}^{n_{i_2}}}](u_{i_1}u_{i_2}+u_{i_1}^2u_{i_2}^2) \otimes \mathbb{F}_2C_{p_1^{n_1}} (\widehat{\overline{a}_{1}^{p^{{j_{1}}}_{1}}}-\widehat{\overline{a}_{1}^{p^{{j_{1}-1}}_{1}}})\otimes \cdots \otimes \mathbb{F}_2C_{p_r^{n_r}} (\widehat{\overline{a}_{r}^{p^{{j_{r}}}_{r}}}-\widehat{\overline{a}_{r}^{p^{{j_{r}-1}}_{r}}}) .$\\
Next, we show that $\{A_{i_1},~A_{i_2},~B_{i_1},~B_{i_2}\}$ is a set of orthogonal idempotents. For this, we see that
\begin{align*}
   A_{i_1}\cdot A_{i_2}&=(\overline{\alpha}~\overline{w}+\overline{\alpha}~\overline{ w}^2)\cdot (\overline{\alpha}^2 ~\overline{w}+\overline{\alpha} ~\overline {w}^2)=\overline{\alpha}^3~\overline{ w}^2+\overline{\alpha}^2 \overline{w}^3+\overline{\alpha}^4 \overline{w}^3+\overline{\alpha}^3 \overline{w}^4\\&=\overline{\alpha}^3~(\overline{w}+\overline{w}^2)+(\overline{\alpha}^2+\overline{\alpha})~\overline{w}^3=2\cdot (\widehat{\overline{a}_1^{p_1^{j_1}}}-\widehat{\overline{a}_1^{p_1^{j_1-1}}}) \cdots (\widehat{\overline{a}_r^{p_r^{j_r}}}-\widehat{\overline{a}_r^{p_r^{j_r-1}}})=0.
\end{align*}
Similarly, we can prove for others.\\
    \end{proof}
 \end{theorem}
 \begin{remark}
    \rm{(\cite[Remark $2.2$]{HMR})} Let $R$ be a ring and $N$ be a nil ideal. Then the following hold:\\
 $1.$ If $\overline{e_1}$ and $\overline{e_2}$ are orthogonal idempotents of $R/N,$ then the corresponding lifted idempotents $e_1$ and $e_2$ of $R$ are also orthogonal.\\
 $2.$ If $\overline{e}\in R/N$ is a primitive idempotent, then the corresponding lifted idempotent $e \in R$ is also primitive.
 \end{remark}
 \begin{theorem} \rm{(\cite[Corollary $4.8$]{HMR})}
    Let R be a commutative ring, $G$ be a commutative group, a be a nilpotent element of index $k$ in $R$, and s be the characteristic of the quotient ring $R/\langle a \rangle$. If $f+\langle a \rangle G$ is an idempotent element of the group ring $(R/\langle a \rangle)G,$ then $f^{s^{k-1}}$ is an idempotent element of the group ring $RG.$ 
 \end{theorem}

\noindent \textit{Proof of Theorem 2.1}
    The proof immediately follows by uplifting primitive idempotents of $\mathbb{F}_2G$ and using the above results.
\qed  \vspace{.25cm}\\
Note that the above set of primitive idempotents is not a complete set of primitive idempotents. However, we can proceed along similar lines to obtain the complete set of primitive idempotents of $RG.$ 
The information on primitive idempotents of $RG$ will be used in subsequent sections.
\section{\textbf{{THE NUMBER OF WORDS IN A CYCLIC CODE}}} \label{6}

In this section, we obtain the number of words in a cyclic code $\mathcal{C}$. Suppose $\mathcal{C}_1=\langle s^ke_{j_1j_2\cdots j_r}\rangle$ with $0 < k < t$. Note that $s^{t-k}\cdot s^k e_{j_1j_2...j_r}=0,$ which implies that $s^ke_{j_1j_2\cdots j_r}$ is not a free basis of $\mathcal{C}_1$ as an $RG$-module. So, instead of computing dimension, here we compute the number of words in $\mathcal{C}$. The following lemma is useful to prove the main Theorem \ref{ab}. 

\begin{lemma}\label{abcd}
    Let $n=p_1^{n_1}p_2^{n_2}...p_r^{n_r}$, where $p_i{'}s$ are distinct odd primes, $n_i{'}s $ are natural numbers, for $1 \leq i \leq r$;
    $2$ be a primitive root of ${p_i^{n_i}}$, $gcd(p_{i}-1,p_{i{'}}-1)=2$, for all $1\leq i \neq i{'}  \leq r$. Then $ord_{p_1^{n_1-j_1}p_2^{n_2-j_2}...p_r^{n_r-j_r}}(2)=\frac{\varphi(p_1^{n_1-j_1}p_2^{n_2-j_2}...p_r^{n_r-j_r})}{2^{r-1}}$, for all $1\leq j_i \leq n_i-1, 1 \leq i \leq r$. 
    \begin{proof}
         We will prove it by induction on $r$.\\
        Suppose $n=p_1^{n_1}p_2^{n_2}$. Let $ord_{p_1^{n_1-j_1}p_2^{n_2-j_2}}(2)=h$. Since $2^h\equiv 1$ $mod~{p_1^{n_1-j_1}p_2^{n_2-j_2}}$, we get $2^h \equiv 1$ $mod~p_1^{n_1-j_1}$ and $2^h \equiv 1$ $mod~p_2^{n_2-j_2}$ . As $2$ is a primitive root of  $p_1^{n_1}$, by  Lemma $1$ of \cite{GKMR}, $2$ is also primitive root of $p_1^{n_1-j_1}$ for $1\leq j_1 \leq {n_1-1}$. Therefore, $\varphi(p_1^{n_1-j_1})|h$. Similarly, $\varphi(p_2^{n_2-j_2})|h$ and hence $lcm({\phi(p_1^{n_1-j_1}),\varphi(p_2^{n_2-j_2})})|h$. Since $gcd(p_1-1,p_2-1)=2$, we have $lcm({\phi(p_1^{n_1-j_1}),\varphi(p_2^{n_2-j_2})})=\frac{{\phi(p_1^{n_1-j_1})\varphi(p_2^{n_2-j_2})}}{2}.$ 
        On the other hand, as $2^{\varphi(p_1^{n_1-j_1})} \equiv1$ $(mod~p_1^{n_1-j_1})$, this implies that  $2^{\varphi(p_1^{n_1-j_1})\varphi(p_2^{n_2-j_2})/2} \equiv1$ $(mod~p_1^{n_1-j_1})$. Similarly,\\ $2^{\varphi(p_1^{n_1-j_1})\varphi(p_2^{n_2-j_2})/2} \equiv1$ $(mod~p_2^{n_2-j_2})$. Thus $h=\frac{\varphi(p_1^{n_1-j_1})\varphi({p_2^{n_2-j_2}})}{2}$.\\
       Suppose that the result holds for $r-1$ factors. We will prove the result for $r$ factors.
        As $gcd(p_i-1,p_{i{'}}-1)=2$, for all $ 1\leq i \neq i{'}  \leq r$, as there exist atmost one $p_i$, say $p_1$, such that $p_1$ is of the form $p_1-1=2^lk_1$, where $k_1$ odd integer, $l\geq 2$ and the remaining $p_i-1$ is of the form $p_i-1=2k_i$, where $k_i$ is an odd integer. Therefore, $gcd((p_1-1)(p_2-1)\cdots (p_{r-1}-1),(p_r-1))=2$. By hypothesis, $ord_{p_1^{n_1-j_1}p_2^{n_2-j_2}...p_{r-1}^{n_{r-1}-j_{r-1}}}(2)=\frac{\varphi(p_1^{n_1-j_1}p_2^{n_2-j_2}\cdots p_{r-1}^{n_{r-1}-j_{r-1}})}{2^{r-2}}$. Since $2$ is a primitive root mod $p_r^{n_r}$, we have $2^{\varphi(p_r^{n_r})}\equiv 1 mod~p_r^{n_r}$ and $2$ is also primitive root of $p_r^{n_r-j_r}$ which implies that $2^{\varphi(p_r^{n_r-j_r})} \equiv 1~ mod~(p_r^{n_r-j_r})$ and $2^\frac{\varphi(p_1^{n_1-j_1}p_2^{n_2-j_2}\cdots p_{r-1}^{n_r-j_r})}{2^{r-2}}\equiv 1~ mod~(p_1^{n_1-j_1}p_2^{n_2-j_2}\cdots p_{r-1}^{n_r-j_r})$. In a similar way, one can prove that $ord_{p_1^{n_1-j_1}p_2^{n_2-j_2}...p_r^{n_r-j_r}}(2)=\frac{\varphi(p_1^{n_1-j_1}p_2^{n_2-j_2}...p_r^{n_r-j_r})}{2^{r-1}}$.
    \end{proof}
\end{lemma}

The main theorem of this section is as follows:

\begin{theorem}\label{ab}
    If $\mathcal{C}_1=\langle s^k e \rangle$, where $e$ is a primitive idempotent of $RG$, then the number of words in $\mathcal{C}_1$ is 
    \[
   |\mathcal{C}_1| =
\begin{cases}
2^{t - k}, & \text{if } e=e_{0\dots 0}, \\
2^{(t - k)(p_{i_1}^{j_{i_1}}-p_{i_1}^{j_{i_1}-1})}, & \text{if } e=e_{0\cdots j_{i_1}\cdots 0},\\
2^\frac{{(t - k)(p_{i_1}^{j_{i_1}}-p^{j_{i_1}-1}_{i_1})(p^{j_{i_2}}_{i_2}-p^{j_{i_2}-1}_{i_2})}}{2}, & \text{if } e=e_{0\cdots {j_{i_1}}\cdots {j_{i_2}}\cdots 0}^{(1)},\\
\cdots \cdots\\
2^\frac{{(t - k)(p_1^{j_1}-p_1^{j_1-1})(p_2^{j_2}-p_2^{j_2-1})\cdots (p_r^{j_r}-p_r^{j_r-1})}}{2^{r-1}}, & \text{if } e=e_{j_1j_2\cdots j_r}^{(1)},\\

\end{cases}
    \]
\end{theorem}

\begin{proof}
If $e=e_{0...0}=\widehat{G}$ then  $|\langle s^ke\rangle|=|{R}s^k|=|\overline{{R}}|^{t-k}=2^{t-k}.$\\
Now assume $e=e_{0\cdots{{j_{i_1}}}\cdots0}$. Note that for any finite subgroup $H$, $[RGs^k]\widehat{H}\cong R(\frac{G}{H})s^k$ and therefore, \[
    \langle s^ke_{0\cdots{{j_{i_1}}}\cdots0}\rangle={R}Gs^ke_{0\cdots{{j_{i_1}}}\cdots0}\cong {R}C_{p_{i_1}^{n_{i_1}}}s^k( {\widehat{a_{i_1}^{p_{i_1}^{j_{i_1}}}}}-{\widehat{a_{i_1}^{p_{i_1}^{j_{i_1}-1}}}} )
\]
Since $\langle a_{i_1}^{p_{i_1}^{j_{i_1}}}\rangle \subseteq \langle a_{i_1}^{p_{i_1}^{j_{i_1}-1}}\rangle $ and $(\widehat{a_{i_1}^{p_{i_1}^{j_{i_1}}}}-\widehat{a_{i_1}^{p_{i_1}^{j_{i_1}-1}}}) \cdot \widehat{a_{i_1}^{p_{i_1}^{j_{i_1}-1}}}=0 $, it implies that
\[
{R}C_{p_{i_1}^{n_{i_1}}}s^k{\widehat{a_{i_1}^{p_{i_1}^{j_{i_1}}}}}={R}C_{p_{i_1}^{n_{i_1}}}s^k( {\widehat{a_{i_1}^{p_{i_1}^{j_{i_1}}}}}-{\widehat{a_{i_1}^{p_{i_1}^{j_{i_1}-1}}}})\oplus {R}C_{p_{i_1}^{n_{i_1}}}s^k{\widehat{a_{i_1}^{p_{i_1}^{j_{i_1}-1}}}},
\]
and so
\[
|{R}C_{p_{i_1}^{n_{i_1}}}s^k{\widehat{a_{i_1}^{p_{i_1}^{j_{i_1}}}}}|=|{R}C_{p_{i_1}^{n_{i_1}}}s^k ( {\widehat{a_{i_1}^{p_{i_1}^{j_{i_1}}}}}-{\widehat{a_{i_1}^{p_{i_1}^{j_{i_1}-1}}}} )| |{R}C_{p_{i_1}^{n_{i_1}}}s^k{\widehat{a_{i_1}^{p_{i_1}^{j_{i_1}-1}}}}|.
\]
This implies $|{R}C_{p_{i_1}^{n_{i_1}}}s^k{\widehat{a_{i_1}^{p_{i_1}^{j_{i_1}}}}}|=|\overline{{R}}|^{(t-k)p_{i_1}^{j_{i_1}}}$ and therefore,
\[
|\langle s^ke_{0\cdots{{j_{i_1}}}\cdots0}\rangle|=|\overline{{R}}|^{(t-k)(p_{i_1}^{j_{i_1}}-p_{i_1}^{j_{i_1}-1})}=2^{(t-k)(p_{i_1}^{j_{i_1}}-p_{i_1}^{j_{i_1}-1})}.
\]
Assume $e=e_{0\cdots{{j_{i_1}}}\cdots{{j_{i_2}}}\cdots 0}$. Also, 
\[
\langle s^ke_{0\cdots{{j_{i_1}}}\cdots{{j_{i_2}}}\cdots0}\rangle={R}Gs^ke_{0\cdots{{j_{i_1}}}\cdots{{j_{i_2}}}\cdots0}\cong{R}C_{p^{n_{i_1}}_{i_1} p^{n_{i_2}}_{i_2}}s^ke_{j_{i_1}j_{i_2}},\]  where \begin{align*}
    e_{j_{i_1}j_{i_2}}=( {\widehat{a_{i_1}^{p_{i_1}^{j_{i_1}}}}}-{\widehat{a_{i_1}^{p_{i_1}^{j_{i_1}-1}}}} )( {\widehat{a_{i_2}^{p_{i_2}^{j_{i_2}}}}}-{\widehat{a_{i_2}^{p_{i_2}^{j_{i_2}-1}}}} )=e_{j_{i_1}j_{i_2}}^{(1)}+e_{j_{i_1}j_{i_2}}^{(2)},
\end{align*} 
\begin{align*}
    e_{j_{i_1}j_{i_2}}^{(1)}=(u_{i_1}u_{i_2}+u_{i_1}^2u_{i_2}^2)^{2^{t-1}}
\end{align*} and
\begin{align*}
    e_{j_{i_1}j_{i_2}}^{(2)}=(u_{i_1}^2u_{i_2}+u_{i_1}u_{i_2}^2)^{2^{t-1}}
\end{align*} where $u_{i_1}$ and $u_{i_2}$ are same as in Theorem $\ref{4}$. Now 
\[
|\langle s^k e_{j_{i_1}j_{i_2}}^{(1)}\rangle |=|{R}Gs^k e_{j_{i_1}j_{i_2}}^{(1)}|=\left |\frac{s^k{R}[x]}{\langle f \rangle}\right |,
\]
where $f$ is an irreducible polynomial of
$x^n-1=\displaystyle\prod_{d|n} \Phi_d(x)$.\\
Note that the number of irreducible factor of the cyclotomic polynomial $\Phi_{p_{i_1}^{j_{i_1}}p_{i_2}^{j_{i_2}}}(x)=\frac{\varphi({p_{i_1}^{j_{i_1}}p_{i_2}^{j_{i_2}}})}{ord_{p_{i_1}^{j_{i_1}}p_{i_{i_2}}^{j_{i_2}}}(2)}=2$ and each monic irreducible polynomial have same degree $\frac{{(p^{j_{i_1}}_{i_1}-p^{j_{i_1}-1}_{i_1}})(p^{j_{i_2}}_{i_2}-p^{j_{i_2}-1}_{i_2})}{2}$ over $\mathbb{F}_2$. Also, by  Hensel’s lemma \rm{(\cite[Lemma $2.4$]{HS})}, which guarantees that  the factorization into a product of pairwise coprime polynomials over $\overline{R}$
lift to such factorization over $R$.\\
Note that $RGe_{j_{i_1}j_{i_2}}^{(1)}{\cong} \frac{R[x]}{\langle f \rangle}$, where $f$ is a irreducible factor of $\Phi_{p_{i_1}^{j_{i_1}}p_{i_2}^{j_{i_2}}}(x)$. Denote $\Gamma=\frac{R[x]}{\langle f \rangle}$.\\ Define a map 
\begin{align*}
   \eta:&~ \Gamma \longrightarrow s^k\Gamma \\& w\longmapsto s^kw
\end{align*}
which is an epimorphism such that ker$\eta=s^{t-k}\Gamma$ and therefore, $\frac{\Gamma}{(s^{t-k})\Gamma}\cong s^k\Gamma$. Also $\frac{\Gamma}{(s^{t-k})\Gamma}\cong \frac{R_i[x]}{\langle \overline{f} \rangle}$, where ${R}_i=\frac{R}{\langle s^{t-k}\rangle}$ and $\overline{f}=f+\langle s^{t-k}\rangle$.
Therefore, 
\[
|s^{k}\Gamma|=\left|\frac{\Gamma}{(s^{t-k})\Gamma}\right|=\left|\frac{R_i[x]}{\langle \overline{f} \rangle}\right|=|{R}_i|^{degf}=\left |\frac{R}{\langle s^{t-k}\rangle}\right|^{degf}=2^{(t-k)degf},
\]where deg${f=\frac{{(p^{j_{i_1}}_{i_1}-p^{j_{i_1}-1}_{i_1}})(p^{j_{i_2}}_{i_2}-p^{j_{i_2}-1}_{i_2})}{2}}$ and hence  \[
|\langle s^k e_{j_{i_1}j_{i_2}}^{(1)}\rangle|=2^{\frac{(t-k)(p^{j_{i_1}}_{i_1}-p^{j_{i_1}-1}_{i_1})(p^{j_{i_2}}_{i_2}-p^{j_{i_2}-1}_{i_2})}{2}}.
\]
As  $RGe_{j_{i_1}j_{i_2}}^{(2)}{\cong} \frac{R[x]}{\langle f_1 \rangle}$, where $f_1$ is a irreducible factor of $\Phi_{p_{i_1}^{j_{i_1}}p_{i_2}^{j_{i_2}}}(x)$ whose degree is same as $f$.
In a similar way one can prove that $|\langle s^k e_{j_{i_1}j_{i_2}}^{(2)}\rangle|=2^{\frac{(t-k)(p^{j_{i_1}}_{i_1}-p^{j_{i_1}-1}_{i_1})(p^{j_{i_2}}_{i_2}-p^{j_{i_2}-1}_{i_2})}{2}}$ .\\
If $e=e_{j_1j_2\cdots j_r}$, then by Theorem \ref{4}, we can split $e$ into $2^{r-1}$ primitive idempotents and one of them is of the form  $e_{j_1\cdots j_r}^{(1)}=(u_1u_2\cdots u_r+u_1^2u_2^2\cdots u_r^2)^{2^{t-1}}$.\\
Here $RGe_{j_1j_2\cdots j_r}^{(1)}\cong\frac{R[x]}{\langle f_2 \rangle}$, where $f_2$ is a irreducible factor of $\Phi_{p_1^{j_1}p_2^{j_2}\cdots p_r^{j_r}}(x)$ and degree of $f_2$ is \\ ${\frac{(p_1^{j_1}-p_1^{j_1-1})(p_2^{j_2}-p_2^{j_2-1})\cdots (p_r^{j_r}-p_r^{j_r-1})}{2^{r-1}}}.$\\
Further, in a similar way, we get $|\langle s^ke_{j_1j_2\cdots j_r}^{(1)}\rangle|=|{R}Gs^ke_{j_1j_2\cdots j_r}^{(1)}|=2^{\frac{(t-k)(p_1^{j_1}-p_1^{j_1-1})(p_2^{j_2}-p_2^{j_2-1})\cdots (p_r^{j_r}-p_r^{j_r-1})}{2^{r-1}}}$.
Therefore, the number of words in $\mathcal{C}$ is  $$|\mathcal{C}|=2^{(t-k)(1+p_{i_1}^{j_{i_1}}-p_{i_1}^{j_{i_1}-1}+\frac{(p_{i_1}^{j_{i_1}}-p_{i_1}^{j_{i_1}-1})(p_{i_2}^{j_{i_2}}-p_{i_2}^{j_{i_2}-1})}{2}+\cdots +\frac{(p_1^{j_1}-p_1^{j_1-1})(p_2^{j_2}-p_2^{j_2-1})\cdots (p_r^{j_r}-p_r^{j_r-1})}{2^{r-1}}) }.$$
\end{proof}

\section {\textbf{MINIMUM WEIGHT OF A CODE}} \label{7}

In this section, we compute the minimum weight of a code $\mathcal{C}_1$. First, we provide the minimum weight for a code $\mathcal{C}_1=\langle s^ke\rangle$ where $e$ is either $e_{0\cdots 0}$ or $e_{0\cdots j_{i_1} \cdots 0}$ and then we give a bound on the minimum weight of a code $\mathcal{C}_1$ for rest of the primitive idempotents.

 \begin{theorem}\label{wt1}
     If $\mathcal{C}_1=\langle s^k e \rangle $, where $e$ is a primitive idempotent of $RG$, then 
     \[ 
     w{(\mathcal{C}_1)}=
     \begin{cases}
   
        |G|, & \text{if } e=e_{0\cdots 0}, \\
         2p^{n_1}\cdots p_{i_1}^{n_{i_1}-j_{i_1}}\cdots p_r^{n_r} & \text{if } e=e_{0\cdots j_{i_1} \cdots 0}.\\
         
     \end{cases}
     \]
 \end{theorem}
 \begin{proof}
 \justifying
    Assume $e=e_{0\cdots 0}$. Let $0 \neq \beta \in \mathcal{C}$ is of the form $  \beta= \mathop{\sum}\limits_{g\in G} \alpha_ggs^k \widehat{G} , \alpha_g \in {R}$, which can be written as $\beta=r\widehat{G}, r\in R$. Therefore, $w(\mathcal{C}_1)=|G|$.\\
    Now, $e=e_{0\cdots j_{i_1} \cdots 0}$. As
    \begin{align*}
        ({a_{i_1}^{p_{i_1}^{j_{i_1}}}}-a_{i_1}^{p_{i_1}^{j_{i_1}-1}})s^ke_{0\cdots{{j_{i_1}}}\cdots0}&=({a_{i_1}^{p_{i_1}^{j_{i_1}}}}-{a_{i_1}^{p_{i_1}^{j_{i_1}-1}}})s^k {\widehat{a_1}}\cdots ( {\widehat{a_{i_1}^{p_{i_1}^{j_{i_1}}}}}-{\widehat{a_{i_1}^{p_{i_1}^{j_{i_1}-1}}}})\cdots {\widehat{a_r}}\\&=(1-a_{i_1}^{p_{i_1}^{j_{i_1}-1}})s^k{\widehat{a_1}}\cdots {\widehat{a_{i_1}^{p_{i_1}^{j_{i_1}}}}}\cdots{\widehat{a_r}}\in RGs^ke_{0\cdots{{j_{i_1}}}\cdots0}.
    \end{align*}
    Then $w(({a_{i_1}^{p_{i_1}^{j_1}}}-a_{i_1}^{p_{i_1}^{j_{i_1}-1}})s^ke_{0\cdots{{j_{i_1}}}\cdots0})=2 p_1^{n_1} \cdots p_{i_1}^{n_{i_1}-j_{i_1}}\cdots p_r^{n_r}$.\\
    Define $H={\langle a_1\rangle}\times{\langle a_2\rangle}\times \cdots \times {\langle a_{i_1}^{p_{i_1}^{j_{i_1}}}\rangle}\times \cdots \times {\langle  a_r\rangle}$. As ${\widehat{a_{i_1}^{p_{i_1}^{j_{i_1}}}}}={\widehat{a_{i_1}^{p_{i_1}^{j_{i_1}-1}}}}+( {\widehat{a_{i_1}^{p_{i_1}^{j_{i_1}}}}}-{\widehat{a_{i_1}^{p_{i_1}^{j_{i_1}-1}}}})$, we get
     \begin{align*}
         (RG)(s^k\widehat{H})=({R}G)s^k{\widehat{a_1}}\cdots{\widehat{a_{i_1}^{p_{i_1}^{j_{i_1}-1}}}}\cdots {\widehat{a_r}} \oplus ({R}G)s^ke_{0\cdots{{j_{i_1}}}\cdots0},
     \end{align*} hence $({R}G)s^ke_{0\cdots{{j_{i_1}}}\cdots0} \subset ({R}Gs^k\widehat{H})$ and $wt({R}Gs^k\widehat{H}) \leq wt({R}Gs^ke_{0\cdots{{j_{i_1}}}\cdots0}) \leq 2 p_1^{n_1} \cdots p_{i_1}^{n_{i_1}-j_{i_1}}\cdots p_r^{n_r} $.\\
      The transversal of $H$ in $G$ is $\mathcal{T}=\{1, a_{i_1},a_{i_1}^2\cdots a_{i_1}^{p_{i_1}^{j_{i_1}-1}}\}.$ Since $RG$ is a free $RH$ module with $\mathcal{T}$ is a free basis and $h\cdot \widehat{H}=\widehat{H}$, this implies that any element of $\beta\in \mathcal{C}$ can be written as $\beta =\beta_0s^k\widehat{H}+ \beta_1s^ka_{i_1}\widehat{H}+\beta_2s^ka_{i_1}^2\widehat{H}\cdots \beta_pa_{i_1}^p\widehat{H}+\cdots \beta_{p_{i_1}^{j_{i_1}-1}} s^ka_{p_{i_1}^{j_{i_1}-1}} \widehat{H}$ and hence $wt(\beta)=r|{H}|,$ with $r\geq 1.$\\ 
 Suppose only one coefficient is non-zero say $\beta_j$, then $\beta=\beta_js^ka_{i_1}^j\widehat{H}$ and there exist $\gamma \in RG$ such that  $\beta=\beta_js^ka_{i_1}^j\widehat{H}=\gamma \cdot(s^ke_{0\cdots {{j_{i_1}}}\cdots 0})$. As $\widehat{H}\cdot \widehat{a_{i_1}^{p_{i_1}^{j_{i_1}-1}}}={\widehat{a_1}}\cdots {\widehat{a_{i_1}^{p_{i_1}^{j_{i_1}-1}}}}\cdots {\widehat{a_r}}$ and $e_{0\cdots {{j_{i_1}}} \cdots 0}\cdot (\widehat{a_{i_1}^{p_{i_1}^{j_{i_1}-1}}})=0$, multiplying $\beta$ by $\widehat{a_{i_1}^{p_{i_1}^{j_{i_1}-1}}}$, we get $\beta_js^ka_{i_1}^j\widehat{a_1}\cdots \widehat{a_{i_1}^{p_{i_1}^{j_{i_1}-1}}}\cdots \widehat{a_r}=0$ which implies that $\beta_js^k=0$ and hence $\beta=0$, that leads  to a contradiction. Thus, $wt(\langle s^ke_{0\cdots j_{i_1}\cdots 0}\rangle)=2 p_1^{n_1} \cdots p_{i_1}^{n_{i_1}-j_{i_1}}\cdots p_r^{n_r}.$

 \end{proof}
 
Note that if $e=e_{0\cdots{{j_{i_1}}}\cdots{{j_{i_2}}}\cdots 0}$ then $e$ can be written as sum of two primitive idempotents $e_{0 \cdots{{j_{i_1}}}\cdots{{j_{i_2}}}\cdots 0}^{(1)}$, $e_{0\cdots{{j_{i_1}}}\cdots{{j_{i_2}}}\cdots 0}^{(2)}$ which are defined in Theorem \ref{4}. The following theorem provide a bounds of a code $\mathcal{C}_1=\langle s^ke_{0\cdots{{j_{i_1}}}\cdots{{j_{i_2}}}\cdots 0}^{(1)}\rangle$. Similarly, one can compute a bounds for the $\mathcal{C}_1=\langle s^ke_{0\cdots{{j_{i_1}}}\cdots{{j_{i_2}}}\cdots 0}^{(2)}\rangle$.
 \begin{theorem} \label{wt2}  If $\mathcal{C}_1=\langle s^k e_{0\cdots{{j_{i_1}}}\cdots{{j_{i_2}}}\cdots 0}^{(1)} \rangle $, then $wt(\mathcal{C}_1) \geq 4p_1^{n_1}\cdots p_{i_1}^{n_{i_1}-j_{i_1}}\cdots p_{i_2}^{n_{i_2}-j_{i_2}}\cdots p_r^{n_r}$.
     
 \end{theorem}

 \begin{proof}
 Since $e=e_{0\cdots{{j_{i_1}}}\cdots{{j_{i_2}}}\cdots 0}^{(1)}+e_{0\cdots{{j_{i_1}}}\cdots{{j_{i_2}}}\cdots 0}^{(2)}$ is a sum of primitive  orthogonal idempotents, this yields that $e_{0\cdots{{j_{i_1}}}\cdots{{j_{i_2}}}\cdots 0}^{(1)},e_{0\cdots{{j_{i_1}}}\cdots{{j_{i_2}}}\cdots 0}^{(2)}\in RGe$. Therefore, $wt(RGs^ke)\leq wt(RGs^ke_{0\cdots{{j_{i_1}}}\cdots{{j_{i_2}}}\cdots 0}^{(1)})$.\\  
     Define $H=\langle a_1\rangle \times \cdots \times \langle a_{i_1}^{p_{i_1}^{j_{i_1}}}\rangle\times \cdots  \times \langle a_{i_2}^{p_{i_2}^{j_{i_2}}}\rangle\times \cdots \times \langle a_r \rangle$. Since $s^ke=(s^ke)\cdot \widehat{H}$, we have $RGs^ke \subset RG(s^k\widehat{H})$.\\  
     Suppose $0\neq \alpha\in RGs^ke$ can be written as $\alpha=\mathop{\sum}\limits_{t\in \mathcal{T}} \alpha_t t s^k\widehat{H}$ with $\alpha_t \in R$ and $\mathcal{T}$ is a transversal of $H$ in $G$.
     If only one of coefficient is non-zero, say, $\alpha_j$, and others are zero then $\alpha=\alpha_jt_js^k\widehat{H}$ and since it is element of $RGs^ke$ then there exist $\gamma \in RG$ such that  $\alpha=\alpha_jt_js^k \widehat{H}=\gamma \cdot (s^ke)$. As $\widehat{H}\cdot \widehat{a_{i_1}^{p_{i_1}^{j_{i_1}-1}}}=\widehat{a_1}\cdots \widehat{a_{i_1}^{p_{i_1}^{j_{i_1}-1}}}\widehat{a_{i_2}^{p_{i_2}^{j_{i_2}}}}\cdots \widehat{a_r}$ and $e\cdot \widehat{a_{i_1}^{p_{i_1}^{j_{i_1}-1}}}=0$, multiplying $\alpha$ by $\widehat{a_{i_1}^{p_{i_1}^{j_{i_1}-1}}}$, we get $(\alpha_jt_j)s^k\widehat{a_1}\cdots \widehat{a_{i_1}^{p_{i_1}^{j_{i_1}-1}}}\widehat{a_{i_2}^{p_{i_2}^{j_{i_2}}}}\cdots \widehat{a_r}=0$ implies $\alpha_js^k=0$, hence $\alpha=0$ which leads to a contradiction. \\
    Now, if only two coefficients are non-zero and others are zero then $\alpha=(\alpha_{j}t_j+\alpha_{j{'}}t_{j'})s^k\widehat{H}$, with $t_j \neq t_j{'}$. \\
    Now, write $H=H_{p_1}\times H_{p_2}$ where $H_{p_1}=\langle a_1\rangle \times \cdots \times \langle a_{i_1}^{p_{i_1}^{j_{i_1}}}\rangle\times \langle a_{i_2-1}\rangle \times \langle a_{i_2+1}\rangle\times   \cdots \times\langle a_r\rangle$, $H_{p_2}=\langle {a_{i_2}^{p_{i_2}^{j_{i_2}}}}\rangle$. Define $H_{p_1}^*=\langle a_1\rangle \times \cdots \times \langle a_{i_1}^{p_{i_1}^{j_{i_1}-1}}\rangle \times \langle a_{i_2-1}\rangle \times \langle a_{i_2+1}\rangle\times  \cdots \times \langle a_r\rangle$,  and $H_{p_2}^*=\langle {a_{i_2}^{p_{i_2}^{j_{i_2}-1}}}\rangle$. Since $\alpha \in RGs^ke$ it implies that $\alpha=\alpha(\widehat{H_{p_1}}-\widehat{H_{p_1}^*})$ and hence $(\alpha_jt_j+\alpha_{j{'}}t_{j{'}})s^k\widehat{H_{p_1}^*}\widehat{H_{p_2}}=0$.\\
  We claim that $supp~(t_js^k\widehat{H_{p_1}^*}\widehat{H_{p_2}}) \cap supp~(t_{j{'}}s^k\widehat{H_{p_1}^*}\widehat{H_{p_2}})=\emptyset$. If $x\in supp~(t_j\widehat{H_{p_1}^*}\widehat{H_{p_2})} \cap supp~(t_{j{'}}\widehat{H_{p_1}^*}\widehat{H_{p_2}}) $ then $t_{j{'}}^{-1}t_j\in H_{p_1}^*H_{p_2}$, so there exists an elements $h_{p_1}{'}\in H_{p_1}^*, h_{p_2}\in H_{p_2}$ such that $t_{j{'}}^{-1}t_j=h_{p_1}{'}h_{p_2}$. Similarly, $(\alpha_{j}t_j+\alpha_{j{'}}t_{j{'}})H_{p_1}H_{p_2}^*=0$ so there exist elements $h_{p_1}\in H_{p_1} $ and $h_{p_2}{'} \in H_{p_2}^*$ such that $t_{j{'}}^{-1}t_j=h_{p_1}h_{p_2}{'}$, we get $h_{p_1}^{-1}h_{p_1}{'}=h_{p_2}{'}h_{p_2}^{-1}\in H_{p_1}^* \cap H_{p_2}^*=\{1\}$, Therefore, $t_j H_{p_1} H_{p_2} =t_{j{'}}H_{p_1} H_{p_2}$, that leads to a contradiction.\\
     Assume $\alpha=(\alpha_{j}t_j+\alpha_{j{'}}t_{j{'}}+\alpha_{j^"}t_{j^"})s^k\widehat{a_1}\cdots \widehat{a_{i_1}^{p_{i_1}^{j_{i_1}}}}\widehat{a_{i_2}^{p_{i_2}^{j_{i_2}}}}\cdots \widehat{a_r}$, multiplying by $(\widehat{a_{i_1}^{p_{i_1}^{j_{i_1}}}}-\widehat{a_{i_1}^{p_{i_1}^{j_{i_1}-1}}})$ and $(\widehat{a_{i_2}^{p_{i_2}^{j_{i_2}}}}-\widehat{a_{i_2}^{p_{i_2}^{j_{i_2}-1}}})$, we get $(\alpha_{j}t_j+\alpha_{j{'}}t_{j{'}}+\alpha_{j^"}t_{j^"})s^k\widehat{a_{i_1}^{p_{i_1}^{j_{i_1}-1}}}\widehat{a_{i_2}^{p_{i_2}^{j_{i_2}}}}=0$ and $(\alpha_{j}t_j+\alpha_{j{'}}t_{j{'}}+\alpha_{j^"}t_{j^"}))s^k\widehat{a_{i_1}^{p_{i_1}^{j_{i_1}}}}\widehat{a_{i_2}^{p_{i_2}^{j_{i_2}-1}}}=0$. Also we have $(\alpha_{j}t_j+\alpha_{j{'}}t_{j{'}}+\alpha_{j^"}t_{j^"})s^kH_{p_1}^*H_{p_2}=0$. Similarly, using the above process, we get that  $t_j H_{p_1} H_{p_2} =t_{j{'}}H_{p_1} H_{p_2}=t_{j^"}H_{p_1} H_{p_2}$, that leads to a contradiction.\\
     We claim that $s^k\alpha=s^k(1-\gamma_1)(1-\gamma_2)\widehat{H}\in RGs^ke$, where $\gamma_1\in \langle a_{i_1}^{p_{i_1}^{j_{i_1}-1}}\rangle \backslash \langle a_{i_1}^{p_{i_1}^{j_{i_1}}}\rangle $ and $\gamma_2\in \langle a_{i_2}^{p_{i_2}^{j_{i_2}-1}}\rangle \backslash \langle a_{i_2}^{p_{i_2}^{j_{i_2}}}\rangle$. Note that $s^k\alpha e=s^k(1-\gamma_1)(1-\gamma_2)\widehat{H}e=s^k\alpha $ and hence $s^k\alpha=s^k\alpha e$ and thus $wt(s^k\alpha)=4|H|=4p_1^{n_1}\cdots p_{i_1}^{n_{i_1}-j_{i_1}}\cdots p_{i_2}^{n_{i_2}-j_{i_2}}\cdots p_r^{n_r}$ hence $ wt(\langle s^ke_{0\cdots{{j_{i_1}}}\cdots{{j_{i_2}}}\cdots 0}^{(1)}\rangle)\geq 4 p_1^{n_1}\cdots p_{i_1}^{n_{i_1}-j_{i_1}}\cdots p_{i_2}^{n_{i_2}-j_{i_2}}\cdots p_r^{n_r}$.
      \end{proof}
      Here we are providing a technique to calculate a bound on the minimum weight of  code $\mathcal{C}=\langle s^ke_{0\cdots j_{i_1}j_{i_2}\cdots j_{i_l}\cdots 0}^{(1)}\rangle $ where $2 \leq l \leq r$.\\
       For $l=2$, take the primitive idempotent of the form $e_{0\cdots j_{i_1}\cdots j_{i_2}\cdots 0}^{(1)}$. Let $K_1=\langle a_{i_1}^{p_{i_1}^{j_{i_1}-1}}\rangle \times \langle a_{i_2}^{p_{i_2}^{j_{i_2}-1}}\rangle$ and $K_2=\langle a_{i_1}^{p_{i_1}^{j_{i_1}}}\rangle \times \langle a_{i_2}^{p_{i_2}^{j_{i_2}}}\rangle$, so $K_2\leq K_1\leq G$ and $\frac{K_1}{K_2} \cong C_{p_{i_1}} \times C_{p_{i_2}}=\langle a_{i_1}^{p_{i_1}^{j_{i_1}-1}}K_2\rangle \times \langle a_{i_2}^{p_{i_2}^{j_{i_2}-1}}K_2\rangle $. Define the map\[
     \sigma : (RK_1)s^k\widehat{K_2} \longrightarrow Rs^k \left (\frac{K_1}{K_2}\right)
     \]
     \[
      a_{i_1}^{p_{i_1}^{j_{i_1}-1}}s^k\widehat{K_2} \longmapsto  a_{i_1}^{p_{i_1}^{j_{i_1}-1}}s^kK_2 \]and \[a_{i_2}^{p_{i_2}^{j_{i_2}-1}}s^k\widehat{K_2} \longmapsto a_{i_2}^{p_{i_2}^{j_{i_2}-1}}s^k K_2\] which is an isomorphism. An arbitrary element $\gamma \in (RK_1)s^k\widehat{K_2}$ is of the form \\$$\gamma=\sum_{l=0}^{p_{i_1}-1} \sum_{z=0}^{p_{i_2}-1} \gamma_{lz}s^ka_{i_1}^{lp_{i_1}^{j_{i_1}-1}}a_{i_2}^{zp_{i_2}^{j_{i_2}-1}}\widehat{K_2}.$$ As $supp~(a_{i_1}^{l_1p_{i_1}^{j_{i_1}-1}}a_{i_2}^{z_1p_{i_2}^{j_{i_2}-1}}s^k\widehat{K_2}) \cap supp~(a_{i_1}^{l_2p_{i_1}^{j_{i_1}-1}}a_{i_2}^{z_2p_{i_2}^{j_{i_2}-1}}s^k\widehat{K_2})=\emptyset $, for $1 \leq l_1,l_2 \leq p_{i_1}-1$, $1 \leq z_1,z_2 \leq p_{i_2}-1$ and $(l_1,z_1) \neq (l_2,z_2)$, we have $l_1p_{i_1}^{j_{i_1}-1}\not \equiv l_2p_{i_2}^{j_{i_2}-1} mod ~p_{i_1}^{j_{i_1}}$ or $z_1p_{i_2}^{j_{i_2}-1}\not \equiv z_2p_{i_2}^{j_{i_2}-1} mod ~p_{i_2}^{j_{i_2}}$. Therefore, \[
     w(\gamma)=w(\sum_{l=0}^{p_{i_1}-1} \sum_{z=0}^{p_{i_2}-1}  \gamma_{lz}s^ka_{i_1}^{lp_{i_1}^{j_{i_1}-1}}a_{i_2}^{zp_{i_2}^{j_{i_2}-1}}\widehat{K_2})
     \]
     \[
     =\sum_{l=0}^{p_{i_1}-1} \sum_{z=0}^{p_{i_2}-1} \gamma_{lz}w(s^ka_{i_1}^{lp_{i_1}^{j_{i_1}-1}}a_{i_2}^{zp_{i_2}^{j_{i_2}-1}}\widehat{K_2})
     \]
     \[
     =\sum_{l=0}^{p_{i_1}-1} \sum_{z=0}^{p_{i_2}-1} \gamma_{lz}w(s^k\widehat{K_2}).
     \] 
     Suppose $\alpha \in  (RG)s^ke_{0\cdots j_{i_1}\cdots j_{i_2}\cdots 0}^{(1)}$. Then $\alpha$ can be written as $$\alpha=\sum_{{x}=0}^{p_{i_1}^{j_{i_1}-1}-1} \sum_{{y}=0}^{p_{i_2}^{j_{i_2}-1}-1}\lambda_{{x}{y}}s^ka_{i_1}^{x}a_{i_2}^{y},$$ where $\lambda_{xy} \in (RK_1)e_{0\cdots j_{i_1}\cdots j_{i_2}\cdots 0}^{(1)}$ and hence $w(\alpha)=\mathop{\sum}\limits_{{x}=0}^{p_{i_1}^{j_{i_1}}-1} \mathop{\sum}\limits_{{y}=0}^{p_{i_2}^{j_{i_2}}-1} w(\lambda_{{x}{y}}s^k)$. Therefore, $w(\alpha) \geq w(\lambda_{xy})$. Thus $w(RGs^ke_{0\cdots j_{i_1}\cdots j_{i_2}\cdots 0}^{(1)})=w((RK_1s^ke_{0\cdots j_{i_1}\cdots j_{i_2}\cdots 0}^{(1)})$.

Similarly, for $l=r$, the idempotent $e_{j_1 j_{2}\cdots j_{r}}$ can be written as a sum of $2^{r-1}$ primitive orthogonal idempotents discussed in Theorem \ref{4}. If $e_{j_1 j_{2}\cdots j_{r}}^{(1)}$ is a primitive idempotent corresponding to  $e_{j_1 j_{2}\cdots j_{r}}$ then $e_{j_1 j_{2}\cdots j_{r}}^{(1)} \in RGs^ke_{j_1 j_{2}\cdots j_{r}}$ which implies $wt(RGs^ke_{j_1 j_{2}\cdots j_{r}})\leq wt(RGs^ke_{j_1 j_{2}\cdots j_{r}}^{(1)})$. Take $H=\langle a_1^{p_1^{j_1}}\rangle \times \cdots \times \langle a_{r-1}^{p_{r-1}^{j_{r-1}}}\rangle \times \langle a_r^{p_r^{j_r}}\rangle $, now write $H=H_{p_r-1}\times H_{p_r}$ where $H_{p_r-1}=\langle a_1^{p_1^{j_1}}\rangle \times \cdots \times \langle a_{r-1}^{p_{r-1}^{j_{r-1}}}\rangle$ and $H_{p_r}=\langle a_r^{p_r^{j_r}}\rangle $. Define $H_{p_{r-1}}^*=\langle a_1^{p_1^{j_1-1}}\rangle \times \cdots \times \langle a_{r-1}^{p_{r-1}^{j_{r-1}-1}}\rangle$, $H_{p_r}^*=\langle {a_r^{p_{r}^{j_r-1}}}\rangle $. In a similar way, we get $wt(\langle s^k e_{j_1\cdots j_r}\rangle)=2^{r-1}p_1^{n_1-j_1}\cdots p_r^{n_r-j_r}$ which implies $wt(\langle s^k e_{j_1\cdots j_r}^{(1)}\rangle) \geq 2^{r-1}p_1^{n_1-j_1}\cdots p_r^{n_r-j_r}.$

To calculate upper bound define $H_1=\langle a_1^{p_1^{j_1-1}}\rangle  \times ...\times \langle a_r^{p_r^{j_r-1}}\rangle $ and $H_2=\langle a_1^{p_1^{j_1}}\rangle  \times...\times \langle a_r^{p_r^{j_r}}\rangle $, so $H_2 \leq H_1 \leq G $ and $\frac{H_1}{H_2}\cong C_{p_1}\times \cdots \times C_{p_r}=\langle a_1^{p_{1}^{j_1-1}} H_2 \rangle \times \cdots \times \langle a_r^{p_{r}^{j_r-1}}H_r \rangle $. By similar approach one can prove, $w(RGs^ke_{j_1 \cdots j_r }^{(1)})=w(RH_1s^ke_{j_1 \cdots j_r }^{(1)})$.

\begin{example}
Let  $G=C_{3^{n_1}}\times C_{5^{n_2}}\times C_{11^{n_3}}=\langle a_1\rangle \times \langle a_2 \rangle \times \langle a_3\rangle $ and $R=\mathbb{Z}_4$. Assume $\mathcal{C}_1=\langle s^ke^{(1)}_{j_1j_20}\rangle$. Then write 
$e_{j_1j_20}=(\widehat{a_1^{3^{j_1}}}-\widehat{a_1^{3^{j_1-1}}})(\widehat{a_2^{5^{j_2}}}-\widehat{a_2^{5^{j_2-1}}})\widehat{a_3}=(e_{j_1j_20}^{(1)}+e_{j_1j_20}^{(2)})\widehat{a_3}$, where $e_{j_1j_20}^{(1)}=(u_1u_2+u_1^2u_2^2)^2\widehat{a_3}$ and $e_{j_1j_20}^{(2)}=(u_1^2u_2+u_1u_2^2)^2\widehat{a_3}$ where $u_1$ and $u_2$ are defined in Theorem \ref{4}. Further note that 
\begin{align*}
    e_{j_1j_20}^{(1)}&=(u_{1}u_{2}+u_{1}^2u_{2}^2)^2\widehat{a_3}
    =(u_1u_2+u_1^2u_2^2+2u_1^3u_2^3)\widehat{a_3}\\&=[\widehat{a_1^{3^{j_1}}}\widehat{a_2^{5^{j_2}}}(a_2^{2\cdot5^{j_1-1}}+2a_2^{5^j}+a_2^{3\cdot 5^{j-1}}+2a_1^{3^{j_1-1}}(a_2^{2\cdot5^{j_1-1}}+2a_2^{5^j}+a_2^{3\cdot 5^{j-1}})+a_1^{2\cdot 3^{j_1-1}}(a_2^{2\cdot5^{j_1-1}}+2a_2^{5^j}+\\&a_2^{3\cdot 5^{j-1}}) + a_2^{2^2\cdot5^{j_1-1}}+2a_2^{5^j}+a_2^{6\cdot 5^{j-1}}+2a_1^{3^{j_1-1}}(a_2^{2^2\cdot5^{j_1-1}}+2a_2^{5^j}+a_2^{6\cdot 5^{j-1}})+a_1^{3^{j_1-1}}(a_2^{2^2\cdot5^{j_1-1}}+2a_2^{5^j}+\\& a_2^{6\cdot 5^{j-1}})+(1-\widehat{a_1^{3^{j_1-1}}})(1-\widehat{a_2^{5^{j_2-1}}}))]\widehat{a_3}.
\end{align*}
Therefore, $wt(e_{j_1j_20}^{(1)})=(15+3^{j_1-1}5^{j_2-1})\cdot 3^{{n_1}-j_1}\cdot 5^{{n_2}-j_2}\cdot 11^{n_3}$. If $K_1=\langle a_1^{3^{j_1}-1}\rangle \times \langle a_2^{5^{j_2}-1}\rangle $ and then  $wt(R K_1 s^k e_{j_1j_20}^{(1)})\leq (15+3^{j_1-1}5^{j_2-1})\cdot 3^{j_1-1}\cdot 5^{j_2-1}\cdot 11^{n_3}.$ Similarly, we can compute $wt(R G s^k e_{j_1j_20}^{(2)})\leq (15+3^{j_1-1}5^{j_2-1})\cdot 3^{j_1-1}\cdot 5^{j_2-1}\cdot 11^{n_3}.$

If the code $ \mathcal{C}_1=\langle s^ke_{j_1j_2j_3}^{(1)}\rangle$ where $e_{j_1j_2j_3}^{(1)}=(u_1 u_2 u_3+u_1^2 u_2^2 u_3^2)^2$ such that $wt(e_{j_1 j_2j_3}^{(1)})=(108+3^{j_1-1}5^{j_2-1}11^{j_3-1})\cdot3^{n_1-j_1}5^{n_2-j_{2}} 11^{n_3-j_3}.$ Let $H_1=\langle a_1^{3^{j_1}-1}\rangle \times \langle a_2^{5^{j_2}-1}\rangle \times \langle a_3^{11^{j_3}-1}\rangle$ then $wt(R H_1 s^k e_{j_1j_2j_3}^{(1)})=wt(R G s^k e_{j_1j_2j_3}^{(1)})\leq (108+3^{j_1-1}5^{j_2-1}11^{j_3-1})\cdot3^{j_1-1}5^{j_{2}-1} 11^{j_3-1}.$

\vspace{0.3cm}
\noindent The following table provides the number of codewords and the minimum weight of the cyclic code $\mathcal{C}_1$ of the length $n$, where $n=3^{n_1}5^{n_2}11^{n_3}$:
\vspace{0.5cm}
\begin{center}
\begin{tabular}{ |c|c|c| } 
\hline
\textbf{Code} & \textbf{Number of words} & \textbf{Weight} \\
\hline
$\langle s^k\widehat{a_1}\widehat{a_2}\widehat{a_3} \rangle$ & $2^{2-k}$ & $|G|$\\
\hline
$\langle s^k(\widehat{a_1^{3^{j_1}}}-\widehat{a_1^{3^{j_1-1}}})\widehat{a_2}\widehat{a_3}\rangle$ & $2^{(2-k)(3^{j_1}-3^{{j_1}-1})}$ & $2\cdot3^{n_1-j_1}5^{n_2}11^{n_3}$\\
\hline
$\langle s^k\widehat{a_1}(\widehat{a_2^{5^{j_2}}}-\widehat{a_2^{5^{j_2-1}}})\widehat{a_3}\rangle$ & $2^{(2-k)(5^{j_2}-5^{{j_2}-1})}$ & $2\cdot3^{n_1}5^{n_2-j_2}11^{n_3}$\\ 
\hline
$\langle s^k\widehat{a_1}\widehat{a_2}(\widehat{a_3^{11^{j_3}_2}}-\widehat{a_3^{11^{j_3-1}}})\rangle$ & $2^{(2-k)(11^{j_3}-11^{{j_3}-1})}$ & $2\cdot3^{n_1}5^{n_2}11^{n_3-j_3}$\\
\hline
$\langle s^k(u_1u_2+u_1^2u_2^2)\widehat{a_3}\rangle $ & $2^{(2-k)(3^{j_1}-3^{{j_1}-1})(5^{j_2}-5^{j_2-1})/{2}}$ &  \\
\hline
$\langle s^k(u_1u_2u_3+u_1^2u_2^2u_3^2)\rangle$ & $2^{{(2-k)(5^{j_2}-5^{j_2-1})(11^{j_3}-11^{{j_3}-1})}/{4}} $ & \\
\hline
\end{tabular}
\end{center}
\end{example}
\section{Acknowledgment}
The first-named author acknowledges the financial support provided by the Ministry of Education, Government of India. The first author also acknowledges the partial support from the FIST program of the Department of Science and Technology, Government of India, Reference No. SR/FST/MS-I/2018/22(C). The second-named author acknowledges the research support of the Department of Science and Technology (INSPIRE Faculty No. DST/INSPIRE/04/2023/001200), Govt. of India. The third-named author acknowledges the financial support of ANRF (File no. MTR/2022/000616).

\end{document}